\pdfoutput=1

\newif\ifarxiv
\arxivtrue

\documentclass[a4paper,11pt]{article}

\usepackage{microtype}

\usepackage{mathtools} 
\mathtoolsset{mathic=true} 

\usepackage{amsmath,amssymb,amsthm}
\usepackage{fullpage}
\usepackage{graphicx}
\usepackage{color}
\usepackage{wrapfig}
\usepackage{xspace}
\usepackage{hyperref}

\newtheorem{theorem}{Theorem}
\newtheorem{lemma}{Lemma}
\newtheorem{proposition}{Proposition}

\newcommand{\R}{\mathbb{R}}
\newcommand{\p}{\mathbb{P}}

\newcommand{\eps}{\varepsilon}
\newcommand{\ah}[1]{\langle #1\rangle}
\newcommand{\Eorth}{E^{\perp}}

\DeclareMathOperator{\grad}{grad}

\newenvironment{denseitems}{\list{$\bullet$}%
  {\labelwidth3em\itemsep0pt\parsep0pt\topsep0.6ex}}{\endlist} 
\newenvironment{densedescription}
               {\list{}{\labelwidth0pt \itemindent-\leftmargin
                   \itemsep0pt\parsep0pt\topsep0.6ex%
                   }}
               {\endlist}

\newcommand{\LINES}{\mathfrak{L}}

\newcommand{\FAM}{\mathcal{F}}
\newcommand{\MANI}{\mathfrak{M}}
\newcommand{\FU}{\zeta}
\newcommand{\G}{\mathcal{G}}
\newcommand{\VOL}{\mathcal{U}}
\newcommand{\NORM}{\eta}
\newcommand{\BOLOID}{\mathcal{B}}
\newcommand{\gORTH}{g^{\perp}}
\newcommand{\FAMO}{\FAM^{\perp}}
\newcommand{\GO}{\mathcal{G}^{\perp}}

\newcommand{\pE}{D}

\newcommand{\VOLV}{\bar{\VOL}}

\newcommand{\conv}{\mathit{conv}}

\newcommand{\iiip}{3$^{\scriptscriptstyle\parallel}$}
\newcommand{\iiix}{3$^{\scriptscriptstyle\times}$}
\newcommand{\ivp}{4$^{\scriptscriptstyle\parallel}$}
\newcommand{\ivx}{4$^{\scriptscriptstyle\times}$}

\let\geq\geqslant
\let\leq\leqslant
\let\setminus\smallsetminus
\let\bd\partial

\makeatletter
\ifx\showkeyslabelformat\@undefined\else
\renewcommand{\showkeyslabelformat}[1]{\normalfont\tiny\ttfamily#1}
\fi
\partopsep\z@ \textfloatsep 10pt plus 1pt minus 4pt
\def\section{\@startsection {section}{1}{\z@}{-3.5ex plus -1ex minus
-.2ex}{2.3ex plus .2ex}{\large\bf}}
\def\subsection{\@startsection{subsection}{2}{\z@}{-3.25ex plus -1ex
minus -.2ex}{1.5ex plus .2ex}{\normalsize\bf}}
\def\@fnsymbol#1{\ensuremath{\ifcase#1\or *\or 1\or 2\or
    3\or 4\or 5\or 6\or 7 \or 8\ or 9 \or 10\or 11 \else\@ctrerr\fi}}
\makeatother

\begin{document}

\title{Lines pinning lines%
  \thanks{
    B.A.\ was partially supported by a grant from the U.S.-Israel
    Binational Science Foundation, by NSA MSP Grant H98230-06-1-0016,
    and by NSF Grant~CCF-08-30691.
    Part of the work of B.A.\ on this paper was carried out while
    visiting KAIST in the summer of 2008, supported by the Brain Korea
    21 Project, the School of Information Technology, KAIST, in 2008.
    The cooperation by O.C.\ and X.G.\ was supported by the INRIA
    \emph{\'Equipe Associ\'ee}~KI.  
    O.C.\ was supported by the Korea Science and Engineering
    Foundation Grant~R01-2008-000-11607-0 funded by the Korean
    government.}}

\author{Boris Aronov%
  \thanks{Department of Computer Science and Engineering, Polytechnic
    Institute of NYU, Brooklyn, New York, USA.
    aronov@poly.edu}
  \and
  Otfried Cheong%
  \thanks{Department of Computer Science, KAIST, Daejeon, Korea.
    otfried@kaist.edu}
  \and
  Xavier Goaoc%
  \thanks{LORIA - INRIA Nancy Grand Est, Nancy, France. goaoc@loria.fr}
  \and
  G\"unter Rote%
  \thanks{Freie Universit\"at Berlin, Institut f\"ur Informatik, 
  Berlin, Germany. rote@inf.fu-berlin.de}
}

\ifarxiv
\date{}
\fi

\maketitle

\begin{abstract}
  A line $\ell$ is a \emph{transversal} to a family $\FAM$ of convex
  polytopes in~$\R^3$ if it intersects every member of~$\FAM$. If, in
  addition, $\ell$ is an isolated point of the space of line
  transversals to $\FAM$, we say that $\FAM$ is a \emph{pinning}
  of~$\ell$.  We show that any minimal pinning of a line by polytopes
  in $\R^3$ such that no face of a polytope is coplanar with the line
  has size at most eight.  If in addition the polytopes are pairwise
  disjoint, then it has size at most six.  
\end{abstract}

\section{Introduction}

A \emph{line transversal} to a family $\FAM$ of disjoint compact
convex objects is a line that meets every object of the family.
Starting with the classic work of Gr\"unbaum, Hadwiger, and Danzer in
the 1950's, \emph{geometric transversal theory} studies properties of
line transversals and conditions for their existence.  There is a
sizable body of literature on line transversals in two dimensions.  Of
particular interest are so-called ``Helly-type'' theorems, such as the
following theorem already conjectured by Gr\"unbaum in 1958, but
proven in this form by Tverberg only in~1989: \emph{If $\FAM$ is a
  family of at least five disjoint translates of a compact convex
  figure in the plane such that every subfamily of size five has a
  line transversal, then $\FAM$ has a line transversal.}  More
background on geometric transversal theory can be found in the
classic survey of Danzer et al.~\cite{dgk-htr-63}, or in the more
recent ones by Goodman et al.~\cite{gpw-gtt-93},
Eckhoff~\cite{e-hrctt-93}, Wenger~\cite{w-httgt-04}, or
Holmsen~\cite{h-rpltfto-08}.

Much less is known about line transversals in three dimensions. Cheong
et al.~\cite{cghp-hhtdus-08} showed a Helly-type theorem for pairwise
disjoint congruent Euclidean balls, Borcea et al.~\cite{bgp-ltdb-08}
generalized this to 
families of arbitrary disjoint balls (with an additional ordering
condition).  On the other hand, Holmsen and Matou{\v
s}ek~\cite{hm-nhtst-04} proved that no such theorem can exist for
convex polytopes: For every $n > 2$ they construct a convex
polytope~$K$ and a family~$\FAM$ of $n$ disjoint translates of~$K$
such that $\FAM$ has no line transversal, but $\FAM \setminus \{F\}$
has a transversal for every $F \in \FAM$.

If a line transversal $\ell$ to a family $\FAM$ cannot move without
missing some~$F\in \FAM$, then we call the
line~$\ell$ \emph{pinned} by~$\FAM$.
Here we consider small continuous
movements of $\ell$ in the vicinity of its current position.
(We obviously exclude a translation of $\ell$ parallel to itself,
which moves the points on the line but does not change the line
as a whole.)
In other words, $\ell$ is
pinned if $\ell$ is an isolated point in the space of line
transversals to~$\FAM$.
We will define an appropriate space $\LINES$ of lines in Section~\ref{sec:prel}.
\begin{wrapfigure}[6]{l}{4.5cm}
  \centerline{\includegraphics{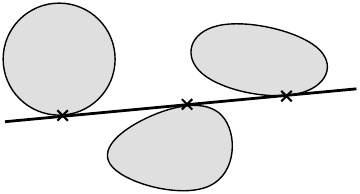}}
\end{wrapfigure}
The line shown on the left is pinned by the three convex figures.
Whenever a family $\FAM$ of convex objects in the plane pins a
line~$\ell$, there are three objects in $\FAM$ that already
pin~$\ell$; this fact was already used by Hadwiger~\cite{h-uegt-57}.
Or, put differently, any \emph{minimal pinning} of a line by convex
figures in the plane is of size three.  Here, a \emph{minimal pinning}
is a family $\FAM$ pinning a line~$\ell$ such that no proper subset
of~$\FAM$ pins~$\ell$.  The Helly-type theorems for balls in three
dimensions are based on a similar result: any minimal pinning of a
line by pairwise disjoint balls in $\R^{3}$ has size at most
five~\cite{bgp-ltdb-08}.

In this paper we prove the analogous result for convex
polytopes, with a restriction:
\begin{theorem}\label{thm:finite}
  Any minimal pinning of a line by possibly intersecting convex
  polytopes in~$\R^3$, no facet of which is coplanar with the line,
  has size at most eight.  The number reduces to six if the polytopes
  are pairwise disjoint.
\end{theorem}

In the light of the construction by Holmsen and Matou{\v s}ek, which
shows that there can be no Helly-type theorem for line transversals to convex
polytopes, Theorem~\ref{thm:finite} is perhaps surprising.
Any Helly-type theorem can be considered as a guarantee for the
existence of a small certificate: Tverberg's theorem, for instance,
says that whenever a family of translates in the plane does not admit
a line transversal, we can certify this fact by exhibiting only five
objects from the family that already do not admit a line
transversal. (This interpretation of Helly-type theorems is the basis
for their relation to \emph{LP-type problems}~\cite{a-httgl-94}.)
Holmsen and Matou{\v s}ek's construction, on the other hand, shows
that there are families $\FAM$ of translates of a convex polytope that
do not admit a line transversal, but such that it is \emph{impossible}
to certify this fact by a small subfamily of~$\FAM$.  Now,
Theorem~\ref{thm:finite} can be interpreted as follows: Whenever there
is a family $\FAM$ of convex polytopes and a line $\ell$ (with the
non-coplanarity restriction) such that there is a neighborhood of~$\ell$
that contains no other line transversal to~$\FAM$, then there is a
subfamily $\G \subset \FAM$ of size at most eight that admits no other
line transversal in a neighborhood of~$\ell$.  So, while the absence
of a line transversal \emph{globally} cannot be witnessed by a small
certificate, \emph{locally} this is the case.

\begin{wrapfigure}[10]{l}{4.5cm}
  \centerline{\includegraphics{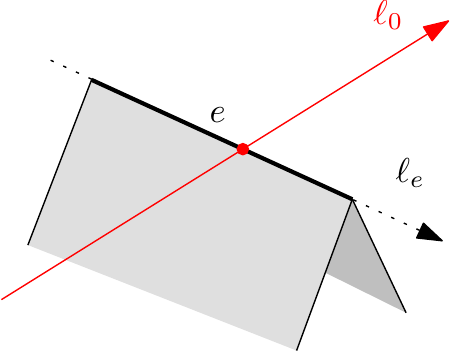}}
\end{wrapfigure}
Consider a family $\FAM$ of convex polytopes pinning a
line~$\ell_{0}$.  A polytope whose interior intersects~$\ell_0$ cannot
contribute to a minimal pinning, so we can assume that each element
of~$\FAM$ is tangent to~$\ell_{0}$. The simplest case is when~$\ell_0$
intersects a polytope~$F$ in a single point interior to an edge~$e$
of~$F$ (see the figure on the left). In that case, the condition that
a line~$\ell$ near~$\ell_0$ intersects~$F$ is characterized by the
``side'' (a notion we formalize in Section~\ref{sec:prel}) on
which~$\ell$ passes the oriented line~$\ell_e$ supporting the
edge~$e$. Since pinning is a local property, it follows that we can
ignore the polytopes that constitute the family~$\FAM$, and speak only
about the lines supporting their relevant edges.  We say that a family
$\FAM$ of \emph{oriented lines} in~$\R^{3}$ pins a line $\ell_{0}$ if
there is a neighborhood of~$\ell_{0}$ such that for any oriented line
$\ell \neq \ell_{0}$ in this neighborhood, $\ell$ passes on the left
side of some line in~$\FAM$.  We will prove the following theorem:
\begin{theorem}
  \label{thm:finite-lines}
  Any minimal pinning of a line~$\ell_{0}$ by
  lines in~$\R^3$ has size at most eight.  If no two of the lines are
  simultaneously coplanar and concurrent with~$\ell_{0}$, then it has
  size at most~six.
\end{theorem}
We will actually give a full characterization of families of lines
that minimally pin a line~$\ell_{0}$ when either all lines are
orthogonal to~$\ell_{0}$, or when they do not contain pairs of lines
that are at the same time concurrent and coplanar with~$\ell_{0}$.

When $\ell_0$ lies in the plane of a facet of a polytope, the
condition that a line close to~$\ell_0$ intersects the polytope
becomes a \emph{disjunction} of two sidedness constraints with respect
to lines.  Thus it cannot be handled by the methods of
Theorem~\ref{thm:finite-lines}.  In fact, the non-coplanarity
condition is necessary: We describe a construction that shows that no
pinning theorem holds for collections of polytopes when they are
allowed both to intersect each other and to touch the line in more
than just a single point:
\begin{theorem}\label{thm:infinite}
  There exist arbitrarily large minimal pinnings of a line by convex
  polytopes in~$\R^3$.
\end{theorem}
We do not know if the condition that the convex polytopes be pairwise disjoint
is
sufficient to guarantee a bounded minimal pinning number.  In fact, we
are not aware of any construction of a minimal pinning by pairwise
disjoint convex objects in~$\R^{3}$ of size larger than six.

\smallskip

Pinning is related to the concept of \emph{grasping} used in robotics:
an object is \emph{immobilized}, or \emph{grasped}, by a collection of
contacts if it cannot move without intersecting (the interior of) one
of the contacts. The study of objects (``fingers'') that immobilize a
given object has received considerable attention in the robotics
community~\cite{mnp-gg-90,m-mrm-01}, and some Helly-type theorems are
known for grasping.  For instance, Mishra et al.'s
work~\cite{MishraSS87} implies Helly-type theorems for objects in two
and three dimensions that are grasped by point fingers.  We give two
examples of how our pinning theorems naturally translate into
Helly-type theorems for grasping a line.  First, we can interpret a
family~$\FAM$ of lines pinning a line $\ell_0$ by considering all
lines as solid cylinders of zero radius.  Each constraint ``cylinder''
touches the ``cylinder''~$\ell_{0}$ on the left. As a result, it is
impossible to move $\ell_{0}$ in any way (except to rotate it around
or translate it along its own axis), because it would then intersect
one of the constraints, so $\ell_0$ is grasped by $\FAM$. Second,
consider a family $\FAM=\{P_1, \ldots, P_n\}$ of polytopes such that
$P_i$ is tangent to $\ell_0$ in a single point interior to an edge
$e_i$, and let $\tilde{P_i}$ denote the mirror image of $P_i$ with
respect to the plane spanned by $\ell_0$ and $e_i$. Since $\FAM$ pins
$\ell_0$ if and only if $\tilde{\FAM}$ grasps it, our pinning theorems
directly translate into Helly-type theorems for grasping a line by
polytopes.


\paragraph*{Outline of the paper.}
Our main technical contributions are our pinning theorem for lines
(Theorem~\ref{thm:finite-lines}), a construction of arbitrarily large
minimal pinnings of a line by overlapping polytopes with facets
coplanar with the line (Theorem~\ref{thm:infinite}), and a
classification of minimal pinnings of a line by lines
(Theorems~\ref{thm:char-ortho} and~\ref{thm:4-pinning}).

Sections~\ref{sec:prel} and~\ref{sec:pin-eight} are devoted to the
proof of the general case of Theorem~\ref{thm:finite-lines}. The idea
of the proof is the following. We consider a parameterization of the
space of lines by a quadric $\MANI \subset \R^5$, where a fixed line
$\ell_0$ is represented by the origin $0 \in \R^5$, and the set of
lines passing on one side of a given line $g$ is recast as the
intersection of $\MANI$ with a halfspace $\VOLV_g$. This associates
with a family $\FAM$ of lines a polyhedral cone~$C$; $C$ represents
all lines not passing on the left side of some element of $\FAM$.
Then $\FAM$ pins $\ell_0$ if and only if the origin $0$ is an isolated
point of $C \cap \MANI$.  We give a characterization of the cones $C$
such that the origin is isolated in $C \cap \MANI$ in terms of the
trace of $C$ on the hyperplane tangent to $\MANI$ in the origin; this
is our \emph{Isolation Lemma} (Lemma~\ref{lem:isolation}).  The
Isolation Lemma allows us to analyze the geometry of the intersection
$C \cap \MANI$ when $\FAM$ is a pinning of $\ell_0$, and to find a
subfamily of size at most eight that suffices to pin.

The subsequent two sections discuss the various configurations of
lines that form minimal pinnings, first for constraints that are
perpendicular to the pinned line (Section~\ref{sec:ortho-minpin}) and
then for the general case (Section~\ref{sec:minpin}).
Theorems~\ref{thm:finite},~\ref{thm:finite-lines},
and~\ref{thm:infinite} are then proven in Section~\ref{sec:wrapup}.

\section{Lines and constraint sets}
\label{sec:prel}

\paragraph*{Sides of lines.} 

Throughout this paper, all lines are oriented unless specified
otherwise.  Given two non-parallel lines $\ell_{1}$ and~$\ell_{2}$
with direction vectors $d_{1}$ and $d_{2}$, we say that $\ell_{2}$
\emph{passes to the right of} $\ell_{1}$ if $\ell_{2}$ can be
translated by a positive multiple of $d_{1}\times d_{2}$ to
meet~$\ell_{1}$, or, equivalently, if
\begin{equation}
  \label{eq:determinant}
  \det \left(\begin{array}{cccc} p_1 & p_1' &p_2 & p_2' 
    \\ 1 & 1 & 1 & 1 \end{array} \right) < 0, 
\end{equation}
where $p_i$ and $p_i'$ are points on $\ell_i$ such that $p_ip_i'$ is a
positive multiple of~$d_i$.

\paragraph*{Constraints.}

 Throughout the paper, $\ell_{0}$ will denote the line to be pinned.
 A line meeting~$\ell_0$ in a single point represents a
 \emph{constraint} on $\ell_0$.  A line~$\ell$ \emph{satisfies} a
 constraint~$g$ if and only if $\ell$ meets $g$ or passes to the right
 of~$g$.  Consequently, the line $\ell_{0}$ is pinned by a
 family~$\FAM$ of constraints if there is a neighborhood of~$\ell_{0}$
 such that $\ell_{0}$ is the only line in this neighborhood satisfying
 all constraints in~$\FAM$.

 \begin{figure}
  \centerline{\includegraphics{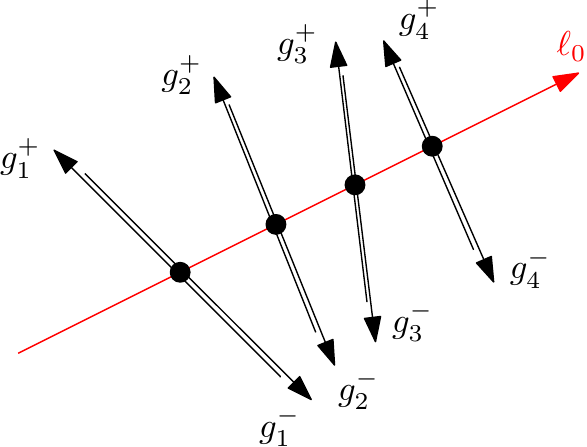}}
  \caption{Eight orthogonal constraints that form a minimal pinning
    of~$\ell_0$.}
  \label{fig:lower-bound-ortho}
 \end{figure}
 We call a constraint \emph{orthogonal} if it is orthogonal
 to~$\ell_0$.  We will give an example of eight orthogonal constraints
 that form a minimal pinning configuration, showing that the constant
 eight in Theorem~\ref{thm:finite-lines} is best possible.  Four
 generically chosen non-oriented lines $g_1, \ldots, g_4$
 meeting~$\ell_0$ and perpendicular to it will have at most two common
 transversals: 
 if no two among $g_1,g_2,g_3$ are coplanar or concurrent, their
 transversals define a hyperbolic paraboloid and it suffices to choose
 $g_4$ not lying on this surface.  Now, let $g_i^+$ and $g_i^-$ denote
 the two oriented lines supported by~$g_i$ (see
 Figure~\ref{fig:lower-bound-ortho}). Since a line satisfies $g_i^+$
 and~$g_i^-$ if and only if it meets~$g_i$, the eight constraints
 $g_1^+, g_1^-, \ldots, g_4^+, g_4^-$ pin~$\ell_0$.  Suppose we remove
 one of the eight constraints, say~$g_1^+$. The common transversals of
 the three remaining lines $g_2,g_3,g_4$
 form a
  hyperbolic paraboloid.
 By our construction, $g_1$ intersects this
  quadric surface
 transversely (since all four lines are orthogonal to~$\ell_{0}$,
  $g_1$ cannot be tangent to the quadric, and since the four lines
 have at most two transversals, $g_1$ cannot lie in the quadric).
 Thus, the quadric of transversals of $g_2,g_3,g_4$ contains
 lines on both sides of $g_1$, and $\ell_0$ is no longer pinned.
 Therefore the eight oriented lines form a minimal pinning, as claimed.

\begin{wrapfigure}[10]{r}{4cm}
\vspace{-\baselineskip}
  \centerline{\includegraphics{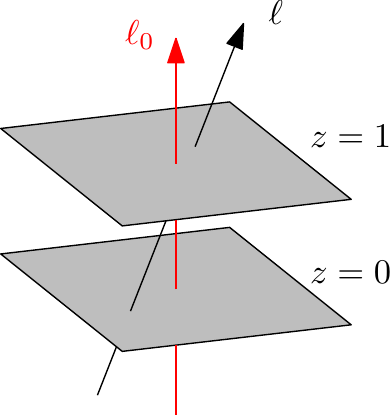}}
\end{wrapfigure}
\paragraph*{A space of lines.}
We choose a coordinate system where~$\ell_{0}$ is the positive
$z$-axis, and denote by $\LINES$ the family of lines whose direction
vector makes a positive dot-product with~$(0,0,1)$.  Since pinning is
a local property, we can decide whether $\ell_{0}$ is pinned
by considering only lines in~$\LINES$ as alternate positions for~$\ell$.
(The constraint lines are of course not restricted to~$\LINES$.)
We identify~$\LINES$ with~$\R^4$
using the intersections of a line with the planes $z=0$ and $z=1$: the
point $u = (u_1,u_2,u_3,u_4) \in \R^{4}$ represents the line $\ell(u)$
passing through the points $(u_1,u_2,0)$ and~$(u_3,u_4,1)$.  The line
$\ell_{0}$ is represented by the origin in~$\R^{4}$.

We describe a constraint using the following three parameters:
$g(\lambda,\alpha, \delta)$ denotes the constraint that meets $\ell_0$
in the point $(0,0,\lambda)$, makes slope $\delta$ with the plane $z =
\lambda$, and projects into the $xy$-plane in a line making an angle
$\alpha$ with the positive $x$-axis. On the line $g(\lambda, \alpha,
\delta)$ we can choose two points $(0, 0, \lambda)$ and $(\cos \alpha,
\sin \alpha, \lambda + \delta)$. 
By~(\ref{eq:determinant}), a line $\ell(u) \in \LINES$ satisfies
$g(\lambda, \alpha, \delta)$ within the space of lines $\LINES$ if
and only if
\[
  \left\lvert 
  \begin{array}{cccc}
    0 & \cos \alpha & u_1 & u_3\\
    0 & \sin \alpha & u_2 & u_4\\
    \lambda & \lambda + \delta & 0 & 1 \\
    1 & 1 & 1 & 1
  \end{array}\right\rvert \leq 0.
\]
The set $\VOL_{g} \subset \R^{4}$ of lines in $\LINES$ satisfying the constraint
$g=g(\lambda, \alpha, \delta)$ is thus 
\begin{equation}\label{eq:defV}
  \VOL_g = \{\, u \in \R^{4}\mid \FU_{g}(u) \leq 0\,\},
\end{equation}
where $\FU_{g}\colon \R^{4} \rightarrow \R$ is defined as
$\FU_{g}(u) = \delta(u_2 u_3 - u_1 u_4) + \NORM_{g} \cdot u$ and
\[
\NORM_g = \NORM(\lambda,\alpha) = 
\left(\begin{array}{c}
  (1-\lambda)\sin\alpha \\
  - (1-\lambda) \cos\alpha \\
  \lambda \sin\alpha \\
  - \lambda \cos\alpha 
\end{array}\right).
\]

\paragraph*{Pinning by orthogonal constraints.} 

As a warm-up, we consider pinning by \emph{orthogonal} constraints. 
If $g$ is an orthogonal constraint, then its parameter
$\delta= 0$, so 
the set~$\VOL_g$ is the halfspace $\NORM_g\cdot u \leq 0$ in $\R^{4}$.
This reduces the problem of pinning by orthogonal constraints to the
following question: when is the intersection of halfspaces in~$\R^4$
reduced to a single point?
It is well-known that this is equivalent to the fact that the
normal vectors surround the origin as an interior point of their
convex hull:
\begin{proposition}
  A finite collection of halfspaces whose boundaries contain the
  origin intersects in a single point if and only if their outer
  normals contain the origin in the interior of their convex hull.
\end{proposition}
Here, the interior of the convex hull refers to the interior
in the ambient space, not the relative interior (in the affine hull).
\begin{proof}
  If the intersection of the halfspaces contains a point $a\ne0$, then
  all normal vectors $n$ must have $n\cdot a\le 0$. This implies that
  all normal vectors $n$, and hence their convex hull, lie in the
  halfspace $\{\,x\mid x\cdot a\le 0\,\}$. Therefore, the convex hull contains no
  neighborhood of the origin~$0$.

  Conversely, if the origin is on the boundary or outside the convex hull,
  there must be a (weakly) separating hyperplane $\{\,a\cdot x= b\,\}$
  such that $a\cdot n\le b$ for all normal vectors $n$, and $a\cdot
  x\ge b$ for $x=0$. This implies $a\cdot n\le 0$ for all normal
  vectors $n$, and thus the point $a\ne0$ lies in all halfspaces.
\end{proof}
To analyze \emph{minimal} pinnings,
we make use of the following classic theorem of Steinitz: 
%
\begin{theorem}[Steinitz]
  \label{thm:steinitz} If a point~$y$ is interior to the convex hull
  of a set $X \subset \R^d$, then it is interior to the convex hull of
  some subset~$Y$ of at most $2d$~points of~$X$. The size of\/~$Y$ is
  at most $2d-1$ unless the convex hull of~$X$ has $2d$~vertices that
  form $d$~pairs $(x,x')$ with $y$ lying on the segment~$xx'$.

  Equivalently, if the intersection of a family~$H$ of halfspaces in
  $\R^{d}$ is a single point, then there is a subfamily of~$H$ of size
  at most~$2d$ whose intersection is already a single point.
\end{theorem}
For a proof of Steinitz's Theorem that includes the second statement,
we refer to Robinson~\cite[Lemma~2a]{robinson-1942}.
Theorem~\ref{thm:steinitz} immediately implies the following lemma.
\begin{lemma}
  \label{lem:ortholine-eight} 
  Any minimal pinning of $\ell_0$ by orthogonal constraints has size
  at most eight.
\end{lemma}
The example of Figure~\ref{fig:lower-bound-ortho} shows that the bound
is tight.  By inspection of the formula for~$\NORM_g$ we observe that
two normals $\NORM_{g_1}$ and~$\NORM_{g_2}$ are linearly dependent if
and only if the orthogonal constraints~$g_1$ and $g_2$ are the same
line up to orientation.  The second statement in
Theorem~\ref{thm:steinitz} implies therefore that the set of the four pairs of
parallel constraints in Figure~\ref{fig:lower-bound-ortho} is in fact
the \emph{only} example of a minimal pinning by eight orthogonal
constraints.

\paragraph*{Linearizing the constraint sets.}

If the constraint $g$ is not orthogonal, then the boundary $\FU_{g}(u)
= 0$ of the set $\VOL_{g}$ is a quadric through the origin.  Since
$\grad \FU_{g}(0) = \NORM_{g}$, the vector $\NORM_{g}$ is the outward
normal of $\VOL_{g}$ at the origin, and we will call it the
\emph{normal} of~$g$. Rather than analyzing the geometry of the
intersection of the volumes~$\VOL_{g}$ bounded by quadrics, we
linearize these sets by embedding~$\LINES$ into five-dimensional space.  This
is based on the observation that the functions $\FU_{g}(u)$ have, up
to multiplication by a scalar, the same nonlinear term
$u_2u_3-u_1u_4$.  The map $\psi\colon \R^{4} \rightarrow \R^{5}$ defined as
\[
\psi\colon (u_1, u_2, u_3, u_4) \mapsto (u_{1}, u_{2}, u_{3}, u_{4},
u_2u_3-u_1u_4) 
\]
identifies $\LINES$ with the quadratic surface 
$$\MANI =\{\, (u_1, u_2, u_3, u_4, u_5)
\mid u_5 =
u_2u_3-u_1u_4\,\} \subset \R^5.$$ The image of $\VOL_g$ under this
transformation is
\begin{equation*}\label{eq:defVb}
  \{\, (u_1, \ldots, u_5)\in 
  \MANI \mid \delta u_5 + \NORM_{g} \cdot (u_1, \ldots, u_4) \leq 0\,\}, 
\end{equation*}
This leads us to define the five-dimensional halfspace
\[
\VOLV_g = \VOLV_{g(\lambda, \alpha, \delta)} = 
\{\,(u_1, \ldots, u_5)\in \R^{5}
\mid \delta u_5 + \NORM_{g} \cdot (u_1, \ldots, u_4) \leq 0\,\} 
.\]
We have $\psi(\VOL_{g}) = \VOLV_{g} \cap \MANI$.  If $\FAM$ is a
family of constraints, then the set of lines in $\LINES$ satisfying all
constraints in $\FAM$ is identified with
\begin{equation}
  \label{eq:cone-repr}
  \bigcap_{g \in \FAM} (\VOLV_{g} \cap \MANI) = 
  \big(\bigcap_{g \in \FAM} \VOLV_{g} \big) \, \cap\, \MANI.
\end{equation}


\paragraph*{Cones and the Isolation Lemma.} 

We restrict our attention to convex polyhedral cones with apex at the
origin, that is, 
intersections of finitely many closed halfspaces whose bounding
hyperplanes go through the origin.  We will simply refer to them as
\emph{cones}.  By~(\ref{eq:cone-repr}), the set of lines satisfying a
family~$\FAM$ of constraints is represented as the intersection of a
cone~$C$ and~$\MANI$.  The family~$\FAM$ pins $\ell_{0}$ if and only
if the origin is isolated in~$C \cap \MANI$.  Our next step is to
characterize cones~$C$ such that the origin is isolated in~$C \cap
\MANI$.

Let $T = \{u_5=0\}$ denote the hyperplane tangent to~$\MANI$ at the
origin.  Let $T^{>} = \{u_{5} > 0\}$ and $T^{<} = \{u_{5} < 0\}$ be
the two open halfspaces bounded by $T$. Similarly, let $\MANI^{>} =
\{u_5 > u_2u_3-u_1u_4 \}$ and $\MANI^{<} = \{u_5 < u_2u_3-u_1u_4 \}$
denote the open regions above and below~$\MANI$. We define $T^{\geq}$,
$T^{\leq}$, $\MANI^{\geq}$, and $\MANI^{\leq}$ analogously; refer to
Figure~\ref{fig:manifold}.  The following lemma
classifies how a ray $\{\,tu\mid t>0\,\}$ starting in the origin may be
positioned with respect to $\MANI$.
\begin{figure}
  \centerline{\includegraphics{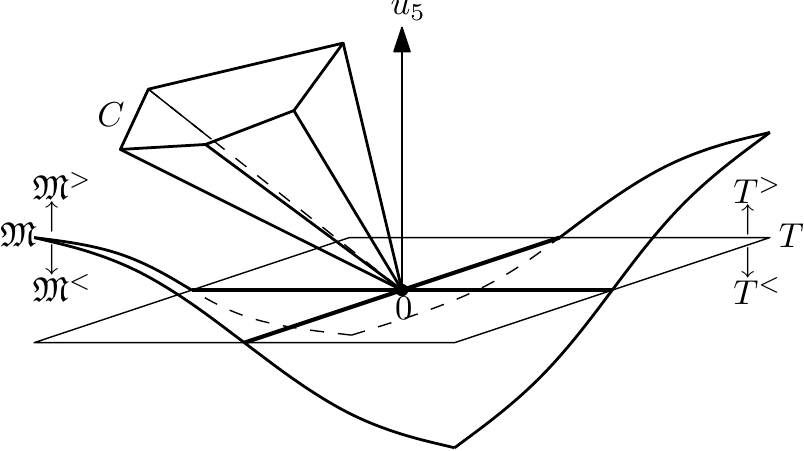}}
  \caption{A symbolic drawing of the manifold $\MANI$ and its tangent
    space~$T$ at~0, and a cone~$C$.  The hyperplane $T\colon
    u_5=0$ is shown as a two-dimensional plane, and the intersection
    $\MANI\cup T=\{\,u_1u_4=u_2u_3\,\}$ is represented as two
    intersecting lines.} 
 \label{fig:manifold}
\end{figure}

\begin{lemma} 
  \label{lem:homogeneous}
  Let $u \in \R^5$ be a non-zero vector.
  \begin{denseitems}
  \item[(i)] If $u \in T\cap\MANI$, then the line
    $\{\,tu\mid t \in \R\,\}$ lies in $\MANI$;
  \item[(ii)] If $u \in T\cap \MANI^>$, then the line $\{\,tu\mid t \in
    \R\,\}$ is contained in $\MANI^> \cup \{0\}$;
  \item[(iii)] If $u \in T\cap \MANI^<$, then the line $\{\,tu\mid t \in
    \R\,\}$ is contained in $\MANI^< \cup \{0\}$;
  \item[(iv)] If $u \in T^{>}$, there is some $\eps > 0$ such that for
    all $t \in (0,\eps)$, $tu \in \MANI^{>}$;
  \item[(v)] If $u \in T^{<}$, there is some $\eps > 0$ such that for
    all $t \in (0,\eps)$, $tu \in \MANI^{<}$.
\end{denseitems} 
\end{lemma}
\begin{proof}
  Since $\MANI$ is a quadric and $T$ is its tangent hyperplane in the
  origin, any line in $T$ through the origin either lies in $\MANI$ or
  meets it only in the origin and otherwise stays completely on
  one side of it.  This can be seen directly from the formula for
  $\MANI\cap T$, and it implies statements~(i)--(iii).  Define now the
  function $f\colon x \mapsto x_{5} - (x_2x_3-x_1x_4)$, whose zero set
  is~$\MANI$; statements~(iv)--(v) follow from the observation that
  $f(tu) = tu_{5} - t^{2}(u_2u_3-u_1u_4)$, and so the sign of~$f(tu)$
  is determined, for small~$t$, by the sign of~$u_{5}$.
\end{proof}
We can now characterize the cones whose intersection with~$\MANI$
contains the origin as an isolated point.

\begin{lemma}[Isolation Lemma]
 \label{lem:isolation}
  Let $C$ be a cone in~$\R^{5}$. The origin is an isolated point of $C
  \cap \MANI$ if and only if either
  \begin{denseitems}
  \item[(i)] $C$ is a line intersecting $T$ transversely, 
  \item[(ii)] $C$ is contained in $T^> \cup (T \cap \MANI^>) \cup
    \{0\}$, or
  \item[(iii)] $C$ is contained in $T^< \cup (T \cap \MANI^<) \cup
    \{0\}$.
  \end{denseitems}
\end{lemma}
\begin{proof}
  Assume that the origin is isolated in $C \cap \MANI$. First, by
  Lemma~\ref{lem:homogeneous}~(i), the intersection of $C$ and $T \cap
  \MANI$ is exactly the origin.  Thus, $C \subseteq T^> \cup (T \cap
  \MANI^>) \cup T^< \cup (T \cap \MANI^<) \cup \{0\}$.  Assume, for a
  contradiction, that $C$ contains both a point $u \in T^{>} \cup (T
  \cap \MANI^{>})$ and a point $v \in T^{<}\cup (T \cap \MANI^<)$, and
  $C$ is not a line.  Since $C$ is not a line, we can ensure that the
  segment $uv$ does not contain the origin, by perturbing $v$ if
  necessary.  By Lemma~\ref{lem:homogeneous}~(ii)--(v), there is an
  $\eps > 0$ such that for $t \in (0,\eps)$, $tu \in \MANI^{>}$ and
  $tv \in \MANI^{<}$. Let $w_t\ne0$ be the point of $\MANI$ on the
  segment joining $tu$ and $tv$; observe that $w_t \in C$, by
  convexity, and $w_t$ tends to~$0$ as $t$ goes to~$0$, which
  contradicts the assumption that the origin is isolated in $C \cap
  \MANI$. The condition is therefore necessary.


  A line intersecting a quadric transversely meets it in at most two
  points, so condition~(i) is sufficient. Assume condition~(ii) holds.
  If $C=\{0\}$ we are done. Otherwise, the set $A = \{ \,  u \in C \mid
  \|u\| = 1\,\}$ is compact and nonempty.  Let $B = \{\, u \in
  \MANI^{\leq}\cap T \mid \|u\| \leq 1\,\}$.  Since $\MANI^{\leq}$ and
  $T$ are closed, $B$ is compact, and since $0 \in B$, it is nonempty.
  Since $C \subset T^> \cup (T \cap \MANI^>) \cup \{0\}$, we have $C
  \cap \MANI^\leq \cap T = \{0\}$, and thus $A$ and $B$ are disjoint.
  Let $\tau > 0$ be the distance between $A$ and~$B$. For any $u =
  (u_1, \dots, u_5) \in C \cap \MANI\setminus\{0\}$, we claim that
  \begin{equation}
    \label{eq:sandwich}
    \tau \|u\| \leq u_{5} \leq \frac12\|u\|^2.
  \end{equation}
  The upper bound follows from $u_5 = u_2u_3-u_1u_4$ and the
  inequality $xy \leq \frac{x^2+y^2}2$ for $x,y \in \R$.  For the
  lower bound, let $v = u/\|u\|$. Since~$C$ is a cone, $v \in C$ and
  so $v \in A$.  By assumption, $u \in C \cap \MANI\setminus\{0\}
  \subseteq T^>$, and thus $u_5>0$.  This implies that the projection
  $u' = (u_{1},\dots,u_{4}, 0)$ of~$u$ on~$T$, is in $\MANI^{<}$.
  Therefore, by Lemma~\ref{lem:homogeneous}~(iii), $v' = u'/\|u\|
  \in \MANI^{<}$.  Since $\|v'\| \leq \|v\| = 1$, we have $v' \in
  B$.  Thus $\|v - v'\| \geq \tau$, and so $u_{5} = \|u -
  u'\| = \|v - v'\|\cdot \|u\| \geq \tau \|u\|$, completing the proof
  of \eqref{eq:sandwich}.  Now, \eqref{eq:sandwich} implies that any
  point $u \in C \cap \MANI$ other than the origin satisfies $\|u\|
  \geq 2\tau$.  This proves that the origin is isolated in $C \cap
  \MANI$. Condition~(ii) is thus sufficient, and the same holds for
  condition~(iii) by symmetry.
\end{proof}

\paragraph*{The tangent hyperplane.} 

 The hyperplane~$T$ tangent to~$\MANI$ in the origin plays a special
 role in the Isolation Lemma.  There is a geometric reason for this:
 $\MANI \cap T$ represents exactly those lines of~$\LINES$ that
 meet~$\ell_{0}$.  Indeed, for $\ell(u)$ to meet $\ell_{0}$, the
 two-dimensional points $(u_{1}, u_{2})$, $(0,0)$, and $(u_{3},
 u_{4})$ have to lie on a line, which is the case if and only if
 $u_{2}u_{3} - u_{1}u_{4} = 0$.  But that is equivalent to~$\psi(u)
\in T$.

\begin{lemma}\label{lem:three-cap-m}
  If $g$ is a constraint, then $T \cap \MANI \cap \bd \VOLV_g$ is
  the union of two two-dimensional linear subspaces.
\end{lemma}
\begin{proof}
  Since $\MANI \cap T$ are the lines in $\LINES$ that meet~$\ell_0$ and $\MANI
  \cap \bd \VOLV_g$ are the lines in $\LINES$ that meet~$g$, $T \cap \MANI
  \cap \bd \VOLV_{g}$ are exactly those lines of~$\LINES$ that
  meet both $\ell_{0}$ and~$g$.  There are two families of such lines,
  namely the lines through the point~$g \cap \ell_{0}$, and the lines
  lying in the plane spanned by $\ell_{0}$ and~$g$.  Each
  family is easily seen to be represented by a two-dimensional linear
  subspace in~$\R^{5}$.
\end{proof}

\paragraph*{Relation to Pl\"ucker coordinates.}
A classic parameterization of lines in space is by means of Pl\"ucker
coordinates~\cite{pottmann-wallner}.  These coordinates map lines in
$\p^3(\R)$ to the Pl\"ucker quadric $\{x_1x_4+x_2x_5+x_3x_6 = 0\}
\subset \p^5(\R)$.  In particular, the oriented line through $(u_1,
u_2, 0)$ and $(u_3, u_4, 1)$ has Pl\"ucker coordinates
\[
(u_3 - u_1, u_4 - u_2, 1, u_2, -u_1, u_1 u_4 - u_2 u_3),
\] 
and so the transformation $\R^{5} \rightarrow \p^{5}(\R)$ defined as
\[
u \mapsto (u_3 - u_1, u_4 - u_2, 1, u_2, -u_1, -u_5)
\]
maps $\psi(u)$ to the Pl\"ucker coordinates of the line~$\ell(u)$.  It
follows that our manifold~$\MANI$ is the image of the Pl\"ucker
quadric under a mapping that sends~$\ell_0$ to the origin, the
hyperplane tangent to the Pl\"ucker quadric at~$\ell_0$
to~$\{u_5=0\}$, and the lines orthogonal to~$\ell_0$ to infinity.  Our
5-dimensional affine representation has the advantage that~$\MANI$
admits a parameterization of the form $u_5= f(u_1, \ldots, u_4)$ where
$f$~is a homogeneous polynomial of degree two. This is instrumental in
our proof of the Isolation Lemma.

We showed that the lines meeting a constraint~$g$ are represented by
(the intersection of $\MANI$ and) a hyperplane in~$\R^{5}$.  This is
no coincidence: the Pl\"ucker correspondence implies that this is true
for the set of lines meeting \emph{any} fixed line~$g$.  Similarly,
given any point $u \in \MANI$, the hyperplane tangent to $\MANI$
at~$u$ intersects $\MANI$ in exactly those lines that meet the line
represented by~$u$---we saw this only for the special hyperplane~$T$.
(Of course, both properties can be easily verified in our setting
without resorting to Pl\"ucker coordinates.)

\section{Minimal pinnings by lines have size at most eight}
\label{sec:pin-eight}

Let $\ah{X}$ denote the linear hull of a set~$X \subset \R^{d}$, that
is, the smallest linear subspace of~$\R^{d}$ containing~$X$.  A
\emph{$j$-space} is a $j$-dimensional linear subspace.  We start with
four lemmas on cones.  (Recall that we defined a cone as the
intersection of halfspaces whose bounding hyperplanes go through the
origin.)

\begin{lemma}
  \label{lem:cone-decomposition}
  Let $C = \bigcap_{h \in H} h$ be a cone defined by a family $H$ of
  halfspaces in~$\R^{d}$. Then
  \[
  \ah{C} = \bigcap_{\underset{\ah{C} \subseteq h}{h \in H}} h = 
  \bigcap_{\underset{\ah{C} \subseteq h}{h \in H}} \bd h.
  \]
\end{lemma}
\begin{proof}
  Clearly $\ah{C} \subseteq \bigcap_{h \in H \mid \ah{C} \subseteq h}
  h$.  To show the reverse inclusion, we first pick an arbitrary
  point~$x$ in the relative interior of~$C$.  Note that for any
   halfspace~$h\in H$,
  $x\in \bd h$ implies $\ah{C} \subseteq h$.

  Let $y \in \bigcap_{h \in H \mid \ah{C} \subseteq h} h$, and
  consider the segment~$xy$.  We show that a neighborhood of~$x$ on
  this segment lies in~$C$, implying $y \in \ah{C}$.  Indeed, consider
  $h \in H$. If $h$ contains~$\ah{C}$ then segment~$xy$ is entirely
  contained in~$h$.  If $h$ does not contain~$\ah{C}$ then $x$ lies in
  the interior of~$h$, and a neighborhood of $x$ on the segment~$xy$
  is in~$h$.

  The second equality follows from the fact that any linear subspace
  (in particular, $\ah{C}$) that is contained in $h$ must also be
  contained in $\bd h$.
\end{proof}
We also use the following extension of Steinitz's Theorem:
\begin{lemma}
  \label{lem:Helly-flat}
  If the cone defined by a family $H$ of halfspaces in~$\R^d$ is a
  $j$-space~$E$, then $2d-2j$ of these halfspaces already define~$E$.
\end{lemma}
\begin{proof}
  Every halfspace $h \in H$ contains~$E$, so its bounding hyperplane
  contains~$E$.  Every $h \in H$ can be decomposed as the Cartesian
  product of~$E$ and $h$'s orthogonal projection on the $(d-j)$-space
  $F$ orthogonal to~$E$.  A subfamily of~$H$ intersects in exactly~$E$
  if and only if their orthogonal projections intersect in exactly the
  origin. Since the projection of $H$ on~$F$ is a collection of
  halfspaces that intersect in a single point, by Steinitz's Theorem
  some $2(d-j)$ of these sets must already intersect in a single
  point, and the statement follows.
\end{proof}
In the next two lemmas we consider cones that lie entirely
in~$T^{\geq}$, or even in $T^{>} \cup \{0\}$.  Since we need these
lemmas for arbitrary dimension, we define $T_d = \{\,x \in \R^{d}\mid
x_d =0\,\}$, $T_d^> = \{\,x \in \R^{d} \mid x_d >0\,\}$, and $T_d^\geq =
T_d^> \cup T_d = \{\,x \in \R^{d} \mid x_d \geq 0\,\}$.
\begin{lemma}
  \label{lem:cone-decomposition2}
  Let $C = \bigcap_{h \in H} h  \subseteq T_d^{\geq}$ be a cone defined
  by a family~$H$ of halfspaces in~$\R^{d}$. Then
  \[
  \bigcap_{\underset{\ah{C \cap T_d} \subseteq h}{h \in H}} h \subseteq T_d^{\geq}.
  \]
\end{lemma}
\begin{proof}
  Assume there is a point $y \in T_d^{<}$ in $\bigcap_{h \in H
    \mid\ah{C \cap T_d} \subseteq h} h$.  Pick a point~$x$ in the
  relative interior of~$C \cap T_d$.  We consider the segment~$xy$,
  and show that a neighborhood of~$x$ on this segment lies in~$C$, a
  contradiction to $C \subset T_d^{\geq}$.  Indeed, if $\ah{C \cap
    T_d} \subseteq h$ we have $y \in h$ and so $xy$ lies entirely
  in~$h$.  On the other hand, if $\ah{C \cap T_d} \not\subset h$, then
  $x$ lies in the interior of~$h$, and a neighborhood of~$x$ lies
  in~$h$.
\end{proof}
\begin{lemma}
  \label{lem:positive-cone}
  Let $C = \bigcap_{h \in H} h$ be a cone defined by a family $H$ of
  halfspaces in~$\R^d$, with $d \geq 2$,
such than no $h \in H$ is bounded by the hyperplane~$T_d$.
 If $C$ is contained in
  $T_d^{>} \cup \{0\}$
 then there is a subfamily $H' \subset H$ of size at
  most $2d -2$ such that the cone $C' = \bigcap_{h \in H'} h$ defined
  by $H'$ is contained in $T_d^{>} \cup \{0\}$.
\end{lemma}
\begin{proof}
 The
  cone $C$ is nonempty, but it does not intersect the hyperplane $F =
  \{x_d = -1\}$. Helly's theorem thus implies that there is a subset
  $H_d\subset H$ of size at most~$d$ such that $C_d\cap F =
  \emptyset$, where $C_d = \bigcap_{h \in H_d} h$. Since $C_d$ is a
  cone, this implies $C_d \subset T_d^{\geq}$. If $C_d \subset T^{>}
  \cup \{0\}$, we are done. Otherwise, let $E = \ah{C_d \cap T_d}$.
Since $C_d \subset
  T_d^\geq$, $C_d \cap T_d$ is a face of~$C_d$. Since no halfspace
  in~$H_d$ is bounded by~$T_d$, this face cannot be
  $(d-1)$-dimensional, and so the dimension~$k$ of~$E$ satisfies $1
  \leq k \leq d - 2$.  For $d = 2$, we already obtain a contradiction,
  establishing the induction basis for an inductive proof.
  So let $d > 2$ and assume that
  the statement holds for dimensions~$2\leq j < d$. 
 Let $K = \{\,h \in H_d \mid E \subset h\,\}$. 
 By
  Lemma~\ref{lem:cone-decomposition} applied to the cone $C_d \cap
  T_d$, we have $E = \bigcap_{h \in K} h \cap T_d$, and by
  Lemma~\ref{lem:cone-decomposition2} we have $\bigcap_{h \in K} h
  \subset T_d^\geq$. For $h \in K$, we define $\hat{h}$ as the
  projection of $h$ on the $(d - k)$-space orthogonal to~$E$.  Let
  $\hat{T}_d^{>}$ be the projection of~$T_d^{>}$.  We have $\bigcap_{h
    \in K} \hat{h} \subset \hat{T}_d^{>} \cup \{0\}$, and since $2\leq d-k <
  d$, the induction hypothesis implies that there is a subset $K'
  \subset K$ of size at most $2(d-k) - 2$ such that $\bigcap_{h \in
    K'} \hat{h} \subset \hat{T}_d^{>} \cup \{0\}$.  This implies
  $\bigcap_{h \in K'} h \subset T_d^{>} \cup E$.  By the original
  assumption, we have $\bigcap_{h \in H} h \cap E = \{0\}$, and
  Steinitz's theorem (Theorem~\ref{thm:steinitz}) inside the
  subspace~$E$ implies that there is a subset $K'' \subset H$ of size
  at most~$2k$ such that $\bigcap_{h \in K''}h \cap E = \{0\}$.
  Setting $H' = K' \cup K''$ we have $\bigcap_{h \in H'} h \subset
  T_d^{>} \cup \{0\}$ with $|H'| \leq 2(d-k) - 2 + 2k = 2d-2$,
  completing the inductive step.
\end{proof}
We are now ready to prove the first half of
Theorem~\ref{thm:finite-lines}:
\begin{lemma}  \label{lem:general-eight}
  Any minimal pinning $\FAM$ of a line by constraints has size at most
  eight.
\end{lemma}
\begin{proof}
  Let $\FAM$ be a minimal pinning of the line~$\ell_0$ by constraints,
  let~$H$ be the family of halfspaces $\{\,\VOLV_{g} \mid g \in \FAM\,\}$,
  and let $C$ denote the cone~$\bigcap_{h\in H} h$. Since $\FAM$ pins
  $\ell_0$, the origin is an isolated point of $C \cap \MANI$ and we
  are in one of the cases~(i)--(iii) of the Isolation Lemma. In
  case~(i), $C$ is a line. By Lemma~\ref{lem:Helly-flat}, $C$ is then
  equal to the intersection of at most eight halfspaces from~$H$,
  implying that $\FAM$ has cardinality at most eight.  Without loss of
  generality, we can now assume that we are in case~(ii) of the
  Isolation Lemma, that is $C \subseteq T^> \cup (T \cap \MANI^>) \cup
  \{0\}$. (Case~(iii) is symmetric.)

  Let $E = \ah{C \cap T}$ and denote the dimension of~$E$ by~$k$.  As
  in the proof of Lemma~\ref{lem:positive-cone}, we observe that
  since~$T$ cannot be the boundary of a halfspace in~$H$, we have $E
  \neq T$ and so $0 \leq k \leq 3$.  Let $H' \subset H$ be the set of
  $h \in H$ with $E \subseteq h$.  For $h \in H'$, we define $\hat{h}$
  as the projection of $h$ on the $(5 - k)$-space orthogonal
  to~$E$.  Let $\hat{T}$, $\hat{T}^{\geq}$, and $\hat{T}^{\leq}$ be
  the projections of $T$, $T^{\geq}$, and $T^{\leq}$, respectively.

  By Lemma~\ref{lem:cone-decomposition2} we have $\bigcap_{h \in H'} h
  \subseteq T^{\geq}$.  Applying Lemma~\ref{lem:cone-decomposition} to
  the cone $C \cap T$ inside the 4-space~$T$, we have $\bigcap_{h \in
    H'} h \cap T = E$.  Together this implies that $\bigcap_{h \in H'}
  \hat{h} \subset \hat{T}^{>} \cup \{0\}$.  Applying
  Lemma~\ref{lem:positive-cone} in the $(5-k)$-dimensional subspace
  orthogonal to~$E$, we have a subfamily $K \subset H'$ of size at
  most $2(5-k) - 2 = 8 - 2k$ such that $\bigcap_{h \in K}\hat{h}
  \subset \hat{T}^{>} \cup \{0\}$, implying that $\bigcap_{h \in K} h
  \subset T^{\geq}$ and $\bigcap_{h \in K} h \cap T = E$.

  We have thus found a small subset~$K$ of constraints that prevent
  $C$ from entering~$T^<$. By case~(ii) of the Isolation Lemma, we
  still have to ensure that the part of the cone that lies within~$T$
  does not enter~$\MANI^{\leq}$, by an appropriate set of additional
  constraints.  We give a direct geometric argument in the
  $k$-dimensional subspace~$E$, arguing separately for each possible
  value of~$k$:
  \begin{itemize}
  \item If $k = 0$, then $E = \{0\}$ and the at most $8 - 0 = 8$~constraints
    in $K$ already pin~$\ell_{0}$.
    
  \item If $k = 1$, then~$E$ is a line contained in $T$. By
    Lemma~\ref{lem:homogeneous}~(ii), that line intersects $\MANI$ in
    a single point, which is the origin.  Thus the at most
    $8 - 2 = 6$ constraints in $K$ already pin~$\ell_{0}$.

  \item If $k = 2$ then $C \cap T$ is a plane, a halfplane, or a
    convex wedge lying in the $2$-space~$E$. We can pick at most two
    constraints $h_1, h_2$ in $H$ such that $(\bigcap_{h \in K} h)
    \cap h_1 \cap h_2 \cap T = C \cap T$. Lemma~\ref{lem:isolation}
    now implies that the at most $8 - 4 + 2 = 6$ constraints $K \cup
    \{ h_{1}, h_{2}\}$ pin~$\ell_{0}$, since their intersection is contained
   in $T^>\cup (T\cap C) \subseteq T^>\cup (T\cap \MANI^>) \cup\{0\}$.

  \item If $k = 3$, applying Lemma~\ref{lem:cone-decomposition} to the
    cone $C \cap T$, we find that $E = \bigcap_{h \in H'} 
    (h \cap T)$. By Lemma~\ref{lem:Helly-flat}, the
    3-space $E$ is the intersection of $T$ and two halfspaces in~$H$.
    This implies that $E$ is of the form $E = T \cap \bd
    \VOLV_{g_0}$ with~$g_0 \in \FAM$. Then, by
    Lemma~\ref{lem:three-cap-m}, $E \cap \MANI$ is the union of two
    2-spaces~$E_{1}$ and~$E_{2}$ that intersect along a
    line~$f$. These two $2$-spaces partition the $3$-space~$E$ into
    four quadrants; since $C \cap T$ intersects $\MANI$ only in the
    origin, it must be contained in one of these quadrants. We project
    $C \cap T$ along~$f$ and obtain a two-dimensional wedge.  The
    boundaries of this wedge are projections of edges of the
    three-dimensional cone $C \cap T$.  Each edge is defined by at
    most two constraints of the three-dimensional cone $C \cap T$, and
    thus we can find at most four constraints $K'$ of $C \cap T$ that
    define the same projected wedge, and thus ensure that $\bigcap_{h
      \in K\cup K'} h \cap T \subseteq (T\cap \MANI^>) \cup\{0\}$.
    Adding these edges to $K$, we obtain a family of at most $8 - 6 +
    4 = 6$ constraints that pin~$\ell_{0}$. \qedhere
  \end{itemize}
\end{proof}
All cases in the proof can actually occur---we will see examples in
Section~\ref{sec:minpin} when we understand the geometry of pinning
configurations better.  The example of
Figure~\ref{fig:lower-bound-ortho} shows that the constant eight is
indeed best possible. However, a look at the proof of
Lemma~\ref{lem:general-eight} shows that in the case were $C \cap T
\neq \{0\}$ (that is, when $k = 1, 2, 3$), a minimal pinning has size
at most six. We will make use of this fact later, in the proof of
Lemma~\ref{lem:general-six}.

The first statement of Theorem~\ref{thm:finite} (that is, the general
case) follows from Lemma~\ref{lem:general-eight}. The reader
interested only in obtaining a finite bound for minimal pinnings by
polytopes can
skip Sections~\ref{sec:ortho-minpin} and~\ref{sec:minpin}.

\section{Minimal pinning configurations by orthogonal constraints}
\label{sec:ortho-minpin} 

In this section, we geometrically characterize minimal pinning
configurations for a line by orthogonal constraints.  We consider our
representation of linespace by~$\R^4$, where the volume of lines
satisfying an orthogonal constraint is a halfspace having the origin
on its boundary.  Recall that, a family of halfspaces through the
origin intersects in a single point if and only if the convex hull of
their normal vectors contains the origin in its interior.  We say that
a set of points in~$\R^{d}$ \emph{surrounds the origin} if the origin
is in the interior of their convex hull.  A set~$N$ \emph{minimally
  surrounds} the origin if $N$ surrounds the origin, but no proper
subset of~$N$ does.

We first give a description of minimal families of points surrounding
the origin in~$\R^4$ as the union of simplices surrounding the origin
in a linear subspace (Theorem~\ref{thm:surrounding}). We then
characterize such simplices that can be realized by normals of
orthogonal constraints (Lemma~\ref{lem:simplices-of-constraints}), and
our classification follows (Theorem~\ref{thm:char-ortho}).

\subsection[Points minimally surrounding the origin in 4-space]
{Points minimally surrounding the origin in~$\R^4$}

  A point set~$S$ \emph{surrounds the
  origin in a linear subspace~$E$} if $S$ spans~$E$ and the origin
lies in the relative interior of~$\conv(S)$.  We note that this is
true if and only if every point $y \in E$ can be written as $y =
\sum_{x\in S} \lambda_{x}x$ with all~$\lambda_{x} \geq 0$.

 A \emph{simplex} of dimension $k$, or \emph{$k$-simplex}, is a set of
 $k+1$ affinely independent points in $\R^d$; we also say
 \emph{segment} for~$k=1$, \emph{triangle} for~$k=2$, and
 \emph{tetrahedron} for~$k=3$.  We say that a simplex~$N$ is
 \emph{critical} if it surrounds the origin in its linear
 hull~$\ah{N_{}}$.  Observe that if $S$ and $T$ are two
 critical simplices then $S \not\subset T$.

\begin{lemma}
  \label{lem:have-critical}
  Let $N$ be a point set with $0 \in \conv(N)$.  Then $N$ contains a
  critical simplex.
\end{lemma}
\begin{proof}
We simply take a simplex $S$ in~$N$ of smallest dimension such that $0 \in
  \conv(S)$. Such a simplex must exist by Carath\'eodory's
  Theorem.
\end{proof}

\begin{lemma}
 \label{lem:overlap}
 Let $A$ be a set that minimally surrounds the origin.
 \begin{denseitems}
  \item[(i)] For any critical simplex $X \subset A$, no point of $A
    \setminus X$ lies in the linear hull of $X$.
  \item[(ii)] If a critical simplex $X \subset A$ contains $k$
    points of a critical $k$-simplex $Y \subset A$ then $X=Y$.
 \end{denseitems}
\end{lemma}
\begin{proof}
(i) Assume that $p \in A \setminus X$ lies in the linear hull of $X$. We
 can then write $p=\sum_{x \in X} \lambda_x x$ with $\lambda_x \geq
 0$.  Let $A'= A \setminus \{p\}$.
 Since $A$ surrounds the origin, every point $q \in \R^{d}$ can be
 written as $q=\sum_{a \in A} \mu_a a$ with $\mu_a \geq 0$. 
 But this can be rewritten as 
 \[
 q = \mu_{p}p + \smashoperator{\sum_{a \in A'}} \mu_{a}a
 = \smashoperator{\sum_{a \in X}}\mu_p \lambda_x x +  
 \smashoperator{\sum_{a \in A'}} \mu_{a}a,
 \]
 and so $q$ can be expressed as a non-negative sum of the elements
 of~$A'$.  But then $A'$ surrounds the origin, and $A$ is not minimal,
 a contradiction that proves statement~(i).

(ii)
 If $Y \subset A$ is a critical simplex then any point $p \in Y$ lies
 in the linear hull of $Y \setminus \{p\}$. Thus, if $X \subset A$ is
 another critical simplex that contains $Y \setminus \{p\}$, then $p
 \in \ah{X}$, and
(i) implies $p\in X$. Thus  $Y \subset X$, and hence $Y=X$.
\end{proof}

Our goal is to describe sets minimally surrounding the origin as
unions of (not necessarily disjoint) critical simplices. Our first
step is the following ``decomposition'' lemma.
\begin{lemma}
  \label{lem:decomposition}
  Let $A$ be a critical simplex of dimension at most~$d-1$, let
  $E=\ah{A}$ be its linear hull, and let $B$ be a set of points
  in~$\R^d$.
  \begin{denseitems}
  \item[(i)] $A \cup B$ surrounds the origin in~$\R^d$ if and only if
    the orthogonal projection of~$B$ on the orthogonal
    complement~$\Eorth$ of~$E$ surrounds the origin in $\Eorth$.
  \item[(ii)] If $A \cup B$ minimally surrounds the origin
  in~$\R^d$ then the orthogonal projection of~$B$ on the orthogonal
  complement~$\Eorth$ of~$E$ minimally surrounds the origin in
  $\Eorth$.
  \item[(iii)] If $\conv(B) \cap E \neq \emptyset$, and $\conv(X)\cap E
    = \emptyset$ for every $X \subsetneq B$, then $B$ is 
    contained in a critical simplex of~$A \cup B$.
  \end{denseitems}
\end{lemma}
\begin{proof}
(i) Let $\pi$ denote the orthogonal projection
  on~$\Eorth$.  If $A\cup B$ surrounds the origin, then every point $y
  \in E^{\perp} \subset \R^{d}$ can be written as $y = \sum_{x \in A
    \cup B} \lambda_{x}x$ with $\lambda_{x} \geq 0$.  Since $\pi(x) =
  0$ for $x\in A$, we have
  \[
  y = \pi(y) = {\sum_{x\in A\cup B}} \lambda_x \pi(x) 
  = {\sum_{x\in B}} \lambda_{x} \pi(x),
  \]
  so $\pi(B)$ surrounds the origin in~$\Eorth$.
  
  Assume now that $\pi(B)$ surrounds the origin in~$\Eorth$, and
  consider a point $y\in \R^{d}$.  Since $\pi(y) \in \Eorth$, we can
  write $\pi(y) = \sum_{x \in B}\lambda_{x}\pi(x)$ with $\lambda_{x}
  \geq 0$.  Let $w = y - \pi(y) + \sum_{x\in B}\lambda_{x}\big(\pi(x) -
  x\big)$. Since $w \in E$, we can express $w$ as $w = \sum_{x\in A}\mu_{x} x$
  with~$\mu_{x} \geq 0$.  Thus
  \[
  y = [y -\pi(y)] + \pi(y) 
  = 
w - \sum_{x\in B}\lambda_{x}(\pi(x) - x)
 + {\sum_{x \in B}}\lambda_{x}\pi(x) 
  = \sum_{x \in A}\mu_x x + \sum_{x \in B}\lambda_x x, 
  \]
  and $A \cup B$ surrounds the origin.
  
(ii) 
If $B$ has a proper subset $C$
  such that $\pi(C)$ already surrounds the origin in $\Eorth$ then
  $A\cup C$ surrounds the origin in $\R^d$ by claim~(i). Thus, if
  $A\cup B$ minimally surrounds the origin in $\R^d$ then $\pi(B)$
  minimally surrounds the origin in~$\Eorth$.

  (iii) Let $\mathcal{A}$ be the set of facets of $A$, and
  let $x$ be an arbitrary point in~$\conv(B) \cap E$.  Then
  $\bigcup_{F\in \mathcal{A}}\conv(F \cup \{x\})$ covers~$\conv(A)$,
  and so $0 \in \conv(F \cup \{x\})$ for some facet $F$ of~$A$. This
  implies $0 \in \conv(F \cup B)$, so by Lemma~\ref{lem:have-critical}
  some simplex $T \subset B \cup F$ is critical.
We will conclude the proof by showing that $B \subset T$.
 Let $X = B \cap T$.
  Since $0 \in \conv(T) \subset \conv(F \cup X)$, we can write $0 =
  \alpha y + \beta z$, where $\alpha, \beta \geq 0$, $y \in
  \conv(F)$ and $z \in \conv(X)$. Since $0 \not\in \conv(F)$ we must
  have $\beta z \neq 0$.  From $0 \in E$ and $\alpha y \in E$, we have
  $\beta z \in E$, which implies $z \in E$, and so $\conv(X) \cap E
  \neq \emptyset$.  By assumption, we now have $X = B$ and
  therefore~$B \subset T$.
\end{proof}

We will now describe any minimal set surrounding the origin in up to
four dimensions as the union of critical simplices. As a warm-up
exercise, we state without proof the following lemma about minimal
sets surrounding the origin in one or two dimensions, since this will
be used in the proof for four dimensions. (The lemma can easily be
proven directly, or along the lines of the proof of
Theorem~\ref{thm:surrounding} below.)
\begin{lemma}
  \label{lem:surrounding-1d2d}
  A point set that minimally surrounds the origin in~$\R^{1}$ consists
  of a positive and a negative point.  A point set that minimally
  surrounds the origin in~$\R^{2}$ is either a triangle with the
  origin as an interior point, or a convex quadrilateral with the
  origin as the intersection point of the diagonals.
\qed
\end{lemma}

\begin{figure}[!htb]
  \centerline{\includegraphics{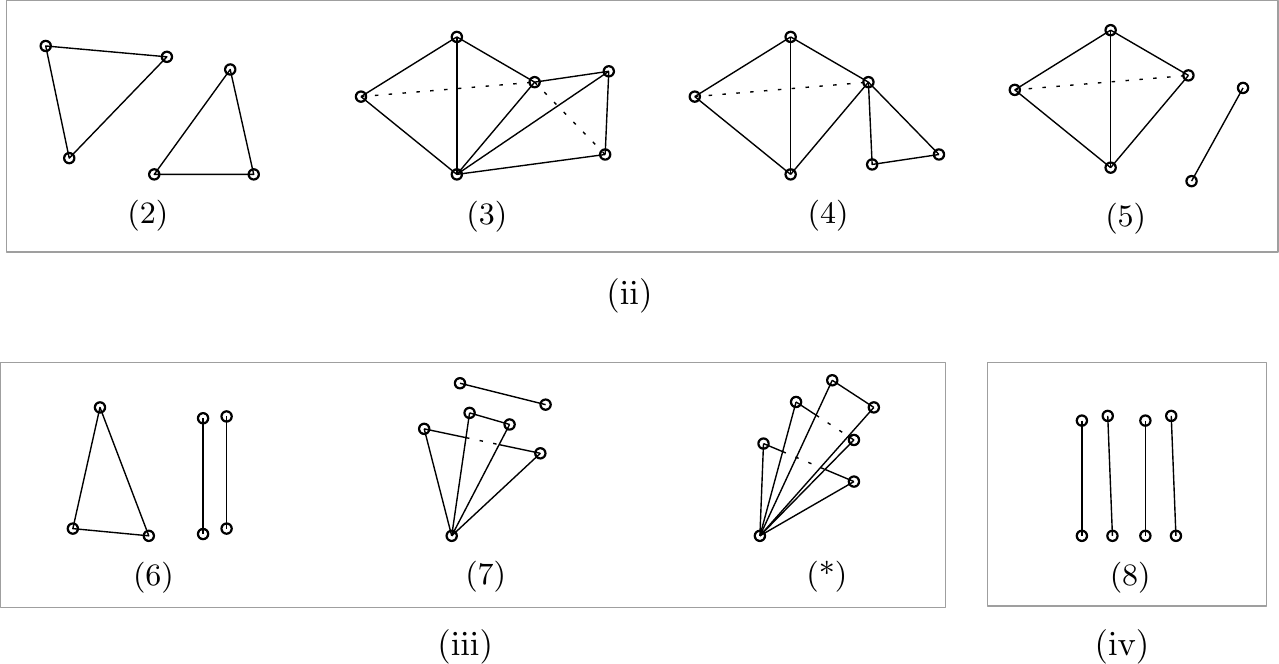}}
  \caption{Combinatorial description of nongeneric minimal sets of
    points surrounding the origin in $\R^4$
according to Theorem~\ref{thm:surrounding}.}
  \label{fig:surrounding}
\end{figure}
We now turn to point sets minimally surrounding the origin
in~$\R^{4}$.  The generic case~(1) is a $4$-simplex surrounding the
origin; the remaining cases~(2)--(8) are depicted in diagram form in
Figure~\ref{fig:surrounding}.  The case~(*) has a special label as we
will see in Theorem~\ref{thm:char-ortho} that it cannot be realized by
normals of orthogonal constraints.  Note that we are not claiming that
the critical simplices shown are \emph{all} critical simplices of the
point set (although we are not aware of a situation that has
additional critical simplices). 
\begin{theorem} \label{thm:surrounding} 
 A set $S$ minimally surrounds the origin in $\R^4$
 if and only if the linear hull of $S$ is $\R^4$ and one of the
 following holds: (i) $|S|=5$ and $S$ is a critical $4$-simplex, or
 (ii) $|S|=6$ and $S$ is the union of two critical simplices, each of
 dimension at most three, or (iii) $|S|=7$ and $S$ is the union of three
 critical simplices: $k\geq 1$ critical triangles having a single
 point in common and $3-k$ disjoint critical segments, or (iv) $|S|=8$
 and $S$ is the disjoint union of four critical segments.
\end{theorem}
\begin{proof}
 We first prove  that 
these cases exhaust all possibilities.
Let $S$ be a minimal
 set of points that surrounds~$0$. By Lemma~\ref{lem:have-critical},
 some simplex of $S$ is critical. Let $A$ be such a simplex with
 maximal cardinality, and denote by $B = S \setminus A$ the remaining
 points of~$S$. Let $E= \ah{A}$ denote the linear hull of $A$ and
 $\Eorth$ the orthogonal complement of~$E$. Let $G$ be the
 \emph{affine} hull of~$B$. The linear hull of $E \cup G$ is $\R^4$,
 as otherwise $S$ cannot surround the origin. We consider various
 cases depending on the cardinality of~$A$.

 If $|A| = 5$, then $A$ surrounds the origin and by minimality of $S$
 we are in case~(i).

 If $|A| = 4$, then $E$ is a 3-space and $\Eorth$ is a line.  By
 Lemmas~\ref{lem:decomposition}~(ii) and~\ref{lem:surrounding-1d2d},
 $B$~consists of exactly two points, one on each side of~$E$. Since $B
 \cap E = \emptyset$ and $\conv(B)$ meets~$E$,
 Lemma~\ref{lem:decomposition}~(iii) implies that $B$ is contained in
 some critical simplex $T \subset S$. Since $A$ was chosen with
 maximal cardinality, we have $|T| \leq 4$, and we are in case~(ii)
 (cases~(3)--(5) of Figure~\ref{fig:surrounding}).

 If $|A| = 2$, any critical simplex of $S$ has size exactly two. It is
 easy to see that there must be exactly four critical segments, and we
 are in case~(iv) (case~(8) of Figure~\ref{fig:surrounding}).

 It remains to deal with the case where $|A| = 3$, so that $E$ is a
 $2$-space. Lemmas~\ref{lem:decomposition}~(ii)
 and~\ref{lem:surrounding-1d2d} imply that $B$ consists of three or
 four points.

 If $|B|=3$, then $B$ is a triangle and $G$ is a $2$-flat that
 intersects~$E$ in a single point~$x$ interior to $\conv(B)$. Since no
 edge of the triangle~$B$ meets~$E$,
 by Lemma~\ref{lem:decomposition}~(iii) $B$ is contained in
 a critical simplex~$T$ of~$S$. Since $|T| \leq 3$ we have $T=B$ and
 $S$ is the disjoint union of two critical triangles. We are thus in
 case~(ii) (case~(2) of Figure~\ref{fig:surrounding}).

 If $|B|=4$, then by Lemma~\ref{lem:surrounding-1d2d} the orthogonal
 projection of $B$ on $\Eorth$ consists of two critical
 segments. Thus, $B$ consists of two disjoint pairs, say $B_1$ and
 $B_2$, whose convex hulls intersect~$E$.  By
 Lemma~\ref{lem:decomposition}~(iii), each $B_i$ is contained in a
 critical simplex of~$S$. Since $A$ is of maximal cardinality, we have
 that $B_i$ either is a critical segment or is contained in a critical
 triangle $T_i = B_i \cup \{a_i\}$.  If at least one $B_i$ is a
 critical segment, then we are in case~(iii) (case~(6) or~(7) of
 Figure~\ref{fig:surrounding}).

 Assume finally that both $B_i$ are contained in a critical
 triangle. If $a_1 \neq a_2$, then $\ah{\{a_1,a_2\}} = E$ and $\ah{T_1
   \cup T_2}=\R^4$.
Since $\ah{T_1}$ and $\ah{T_2}$ are two-dimensional, $\ah{T_1} \cap \ah{T_2} = \{0\}$. It follows
 that the orthogonal projection of $T_2$ on $\ah{T_1}^{\perp}$
is two-dimensional and therefore
 surrounds the origin in that
 subspace. Lemma~\ref{lem:decomposition}~(i) then implies that $T_1
 \cup T_2 \subsetneq S$ surrounds the origin, contradicting the minimality
 of $S$. Hence, $a_1=a_2$ and $S$ is the union of three critical
 triangles with exactly one common point, and we are in case~(iii)
 (case~(*) of Figure~\ref{fig:surrounding}).

 \bigskip
 We now turn to the converse.  We first argue that every point set~$S$
 according to one of the types~(i)--(iv) 
 must surround the origin because
 $\ah{S} = \R^{4}$ and $S$ is the union of critical simplices:
  If $S$ does not surround the origin, then there
 is a closed halfspace~$h$ through the origin such that $S \subset h$.
 Since $\ah{S} = \R^{4}$, there must be a point $p \in S$ in the
 interior of~$h$.  But then the critical simplex $T \subset S$
 containing~$p$ cannot surround the origin in~$\ah{T}$,
a contradiction.

 We proceed to argue the minimality of~$S$:
no proper subset of $S$ 
 surrounds the origin.
We
 distinguish the four cases of the theorem.

 In case~(i), $|S| = 5$ and $S$ is a critical $4$-simplex. By
 definition it surrounds the origin and none of its faces does.

 In case~(iv), $|S| = 8$ and $S$ is the disjoint union of four
 critical segments.  Since $\ah{S} = \R^4$, the directions of these
 segments are linearly independent, and so $S$ minimally surrounds the
 origin.

 In case~(ii), $|S| = 6$. Assume that there is a subset $R \subsetneq
 S$ that surrounds the origin.  Since $|R| \leq 5$, our necessary
 condition implies that $R$ is a 4-simplex. Let $p$ be the point in
 $S$ not in~$R$. Since $S$ is of type~(ii), it can be written as $S=A
 \cup B$ where $A$ and $B$ are critical simplices of dimension at most
 $3$. Since neither $A$ nor $B$ can be contained in $R$, they both
 contain $p$ and are not disjoint. Without loss of generality, we can
 assume that $A$ is a tetrahedron and $B$ is a triangle or a
 tetrahedron. Let $A'=A\setminus \{p\}$ and $B'=B \setminus \{p\}$ and
 consider the ray $r$ originating in $0$ with
 direction~$\vec{p0}$. Since $A=A'\cup \{p\}$ is a critical simplex,
 $r$ intersects the relative interior of $\conv(A')$ in a point~$x_A$;
 similarly, $r$ intersects the relative interior of $\conv(B')$ in a
 point~$x_B$. Put $E=\ah{A}$ and consider the situation in the
 $3$-space $E$ (see Figure~\ref{fig:minimality}).
 \begin{figure}[!htb]
   \centerline{\includegraphics{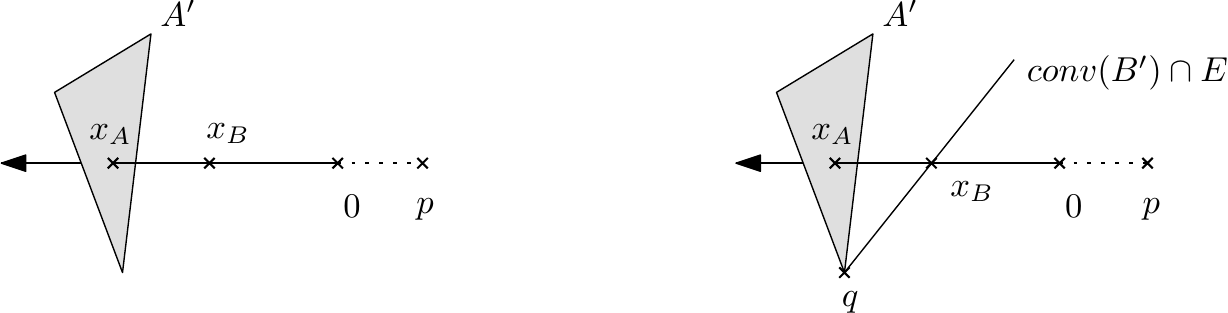}}
   \caption{The situation in $E$ in the case where $x_B$ is closer to
     $0$ than $x_A$ 
     (the other case is symmetric): (left) when $B$ is a triangle,
     (right) when $B$ is a tetrahedron.}  \label{fig:minimality}
 \end{figure}
 There exists a $2$-plane $\Pi$ that strictly separates, in $E$, $A'
 \cup \{x_B\}$ from the origin. The affine hull of $B'$ is either a
 line intersecting $E$ in $x_B$ (if $B$ is a triangle) or a $2$-plane
 intersecting $E$ in the line $qx_B$ (if $B$ is a tetrahedron), where
 $\{q\}=A' \cap B'$. (Indeed, if the intersection of the affine hull
 of $B'$ with $E$ has higher dimension, then $B' \subset E$ and
 $E=\ah{R}$, contradicting our assumption that $R$ surrounds the
 origin.) In either case, there exists a $3$-space in $\R^4$ that
 strictly separates the origin from $A' \cup B' = R$, so $R$
 cannot surround the origin, completing the argument in case~(ii).

 Finally, in case~(iii) $|S| = 7$.  Then $S$ can be written as $S= A
 \cup B \cup C$, where $\{A,B,C\}$ consists of $k\geq 1$ critical
 triangles sharing exactly one vertex~$p$ and $3-k$ disjoint critical
 segments. Assume that there exists a proper subset $R$ of $S$ that
 minimally surrounds the origin.  First consider the case where $p \in
 R$, and note that this implies that any point in $S \setminus R$
 belongs to a unique simplex from $\{A,B,C\}$. If $R$ has size~$5$ it
 contains one critical simplex from $\{A,B,C\}$, and that contradicts
 our necessary condition that $R$ be a critical $4$-simplex. Thus $R$
 must have size $6$ and contains two critical simplices from
 $\{A,B,C\}$, say $A$ and $B$. Our necessary condition implies that
 $R$ can be written $R=X \cup Y$ where $X$ and $Y$ are critical
 simplices of dimension at most $3$. Since $|A| \leq 3$, $X$ or $Y$
 must contain $|A|-1$ points from~$A$.  Without loss of generality let
 this be~$X$.  Lemma~\ref{lem:overlap}~(ii) implies that $X=A$. Then
 $Y$ contains all points from $B$ except possibly~$p$, and so $Y=B$ by
 Lemma~\ref{lem:overlap}~(ii). Since $A \cup B$ has size at most $5$,
 whereas $X \cup Y = R$ has size $6$, we get a contradiction.  Now
 consider the case where $p \not\in R$. Let $r$ be the ray originating
 in $0$ with direction vector $\vec{p0}$. Let $Z \in \{A,B,C\}$ and
 observe that $Z'=Z \cap R$ is either a single point, a critical
 segment, or a non-critical segment. Moreover, if $Z'$ is a
 non-critical segment then $Z=Z' \cup \{p\}$ is a triangle and since
 $Z$ is critical, $r$ intersects the relative interior of
 $\conv(Z')$. Let $G$ be a $3$-space passing through $0$, parallel to
 the segments in $\{A',B',C'\}$ and passing through the points in
 $\{A',B',C'\}$ (if any). The halfspace in $\R^4$ bounded by $G$ and
 containing $r$ has $0$ on its boundary and contains $R$. This implies
 that $R$ cannot surround $0$, a contradiction.
\end{proof}

Finally, the minimally surrounding sets in~$\R^{3}$ can easily be
obtained from Theorem~\ref{thm:surrounding}.
\begin{lemma}
  \label{lem:surrounding-3d}
  If a set~$S\subset \R^{3}$ minimally surrounds the origin
  in~$\R^{3}$ then either (i)~$S$ is a critical tetrahedron, or
  (ii)~$S$ is the union of two critical triangles sharing one point,
  or (iii)~$S$ is the disjoint union of a critical triangle and a
  critical segment, or (iv)~$S$ consists of three disjoint critical
  segments.
\end{lemma}
\begin{proof}
  $S$ minimally surrounds the origin in~$\R^{3}$ if and only if $S$
  can be extended to a point set minimally surrounding the origin
  in~$\R^{4}$ by adding a single critical segment in a transverse
  direction.  We thus obtain all possible configurations from
  Figure~\ref{fig:surrounding}, by taking all configurations
  containing a critical segment, and deleting it.  Cases~(5),~(6),~(7),
  and~(8) of Figure~\ref{fig:surrounding} are possible, and lead to
  the four cases of the lemma.
\end{proof}


\subsection{Critical simplices formed by constraints}

We now discuss how critical simplices can be formed by the normal
vectors of a family of constraints.  A necessary condition for a
family of normals to form a critical simplex is that they be linearly
dependent, but that any proper subset be linearly independent. 

We call a family of constraints \emph{dependent} if their normal
vectors are linearly dependent.  Our first step in rewriting
Theorem~\ref{thm:surrounding} in terms of constraints is to give a
geometric interpretation of dependent families of constraints.

We say that four or more lines $\ell_1, \ldots, \ell_k$ are in
\emph{hyperboloidal position} if they belong to the same family of
rulings of a nondegenerate quadric surface.  In the case of orthogonal
constraints, which are parallel to a common plane, that quadric
surface is always a hyperbolic paraboloid.
\begin{lemma}
  \label{lem:ortho-dependent}
  Two orthogonal constraints are dependent if and only if they 
 are identical except possibly for their orientation.
  
  Three orthogonal constraints are dependent if and only if (i) two of
  them are, or if (ii) the constraints are coplanar with~$\ell_{0}$,
  or if (iii) the constraints are concurrent with~$\ell_{0}$.

  Four orthogonal constraints are dependent if and only if (i) three of them are,
  or if (ii) the constraints are in hyperboloidal position, or if
  (iii) two of the lines are concurrent with~$\ell_{0}$ and the other
  two are coplanar with~$\ell_{0}$.
\end{lemma}
\begin{proof}
  Let $g_{1}, \dots, g_{k}$ be orthogonal constraints,  for $k
  \in \{2,3,4\}$. The normals $\NORM_{g_1},\dots, \NORM_{g_k}$ are
  linearly dependent if and only if they span a linear subspace~$E$ of
  dimension less than~$k$.  This is equivalent to saying that the
  orthogonal complement~$F$ of~$E$ has dimension larger than~$4-k$.
  Since $F = \{\, u\in \R^{4}\mid \FU_{g_{i}}(u) = 0 \text{~for~} 1 \leq
  i \leq k\,\}$, $F$ is the space of lines in~$\LINES$ that meet all
  $k$~constraints~$g_{1},\dots,g_{k}$.  If the normals are linearly
  independent, then $F$ has dimension~$4-k$. If any $k-1$ normals are
  linearly independent, but all $k$ are linearly dependent, then any
  $k-1$ normals already span the subspace~$E$, and its complement
  is~$F$. This is equivalent to saying that every constraint must meet
  all the lines meeting the remaining $k-1$ constraints.

  Consider first two orthogonal constraints~$g_{1}$ and~$g_{2}$. They
  are dependent if and only if every line meeting $g_{1}$ also
  meets~$g_{2}$.  This happens if and only if $g_{1}$ and $g_{2}$ are
  the same line except possibly for their orientation.

  Consider now three orthogonal constraints~$g_{1}, g_{2}, g_{3}$, and
  assume that no two are dependent. Then all three are dependent if
  and only if every line meeting two constraints also meets the third.
  This is impossible if the constraints are pairwise skew.  If two are
  coplanar with $\ell_{0}$, say $g_{1}$ and $g_{2}$, then their common
  transversals are exactly the lines in this plane.  For $g_{3}$ to
  meet all of them, $g_{3}$ has to lie in the same plane as well.  Finally, if
  two constraints, say $g_{1}$ and $g_{2}$, meet in a point $p \in \ell_{0}$ then
  the common transversals are exactly the lines through~$p$. For
  $g_{3}$ to meet all of them, $g_{3}$ has to contain $p$ as well, and
  the three lines are concurrent with~$\ell_{0}$.

  Finally, consider four orthogonal constraints $g_1,\dots,g_4$, and
  assume that no three are dependent. Then every three constraints
  have a one-dimensional family of common transversals.  The four
  constraints are dependent if and only if every constraint meets
  every transversal to the other three. If the lines are pairwise
  skew, then $g_{4}$ must lie in the hyperbolic paraboloid formed by
  the transversals to $g_{1}, g_{2}, g_{3}$, and the four lines are in
  hyperboloidal position.  Otherwise, two constraints must be coplanar
  with $\ell_{0}$ or must meet on $\ell_{0}$.  If two constraints, say
  $g_{1}$ and $g_{2}$, meet in $p \in \ell_{0}$ then the remaining
  constraints cannot contain~$p$. The set of transversals to $g_{1},
  g_{2}, g_{3}$ is the set of lines through~$p$ and~$g_{3}$.  For
  $g_{4}$ to meet all these lines, $g_{4}$ has to lie in the plane
  spanned by $p$ and $g_{3}$, and so $g_{3}, g_{4}, \ell_{0}$ are
  coplanar.  If no two constraints meet, then two must be coplanar,
  say $g_{1}$ and $g_{2}$, and the constraints $g_{3}$ and $g_{4}$
  intersect the plane spanned by $g_{1}, g_{2}, \ell_{0}$ in two
  distinct points $p_{1}$ and $p_{2}$.  But there is only one line
  through $p_{1}$ and $p_{2}$, contradicting the assumption that
  $g_1,\ldots,g_4$ are dependent, and so this case cannot occur.
\end{proof}

\begin{figure}[!htb]
  \centerline{\includegraphics{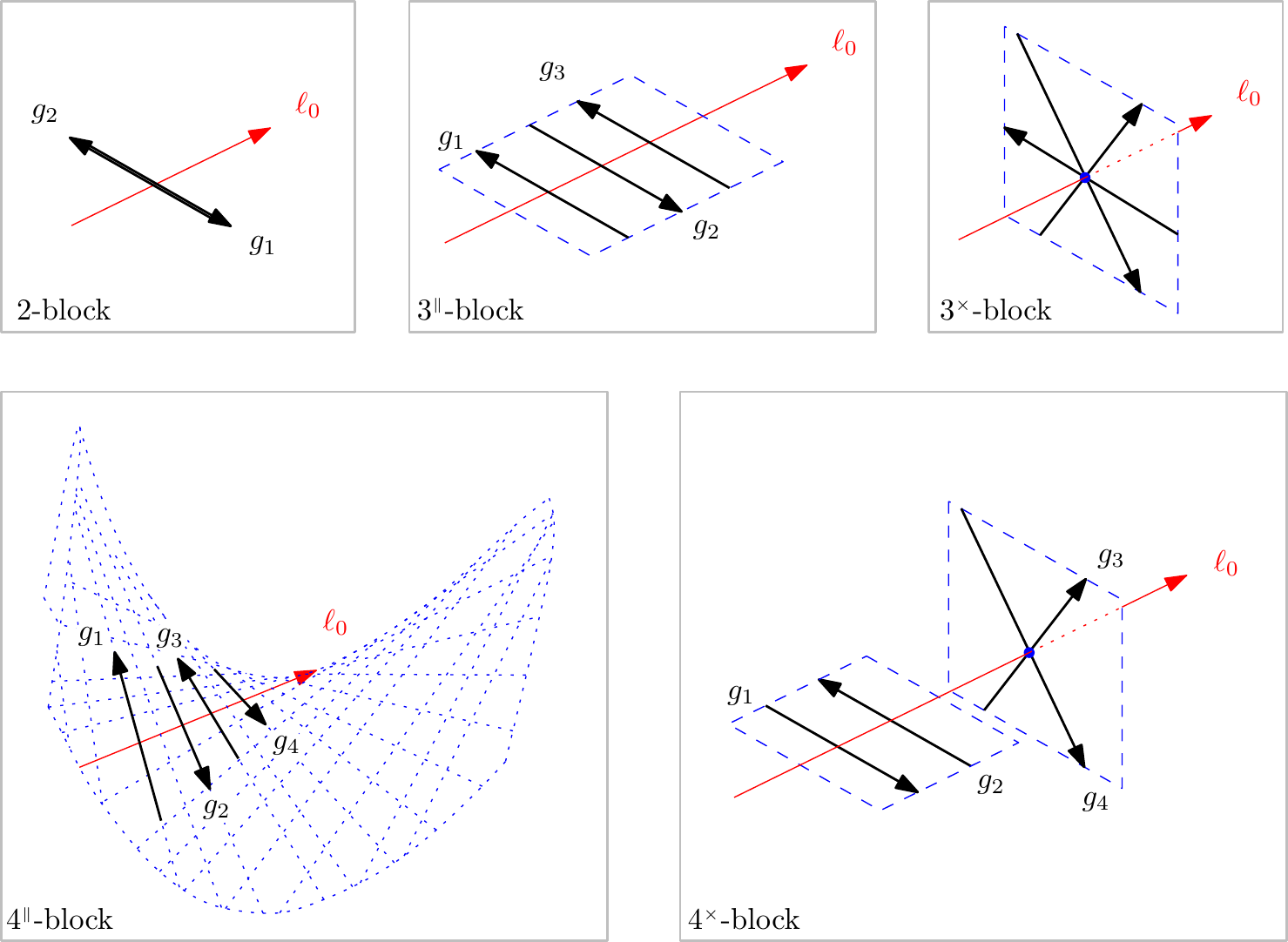}}
  \caption{All types of blocks of size at most four.}
  \label{fig:all-configurations}
\end{figure}
We call a family of constraints a \emph{block} if their normals
form a critical simplex. (We will use blocks to build pinning
configurations.)  The following lemma characterizes blocks
geometrically and introduces names for the various types: A $k$-block
consists of $k$~constraints. \iiix-blocks and \ivx-blocks contain
lines that intersect, while the lines in \iiip-blocks and \ivp-blocks
do not intersect.  Figure~\ref{fig:all-configurations} shows all
blocks of size at most four.
\begin{lemma}
  \label{lem:simplices-of-constraints}
  A family of orthogonal constraints forms a block if and only if it
  contains no proper dependent subfamily and the constraints form one
  of the following configurations:
  \begin{densedescription}
  \item[2-block] the two orientations of the same line;
  \item[\iiip-block] three coplanar constraints with alternating
    orientations (so the lines are $g(\lambda_1, \alpha, 0)$,
    $g(\lambda_2, \alpha+\pi, 0)$ and $g(\lambda_3, \alpha, 0)$ with
    $\lambda_1 < \lambda_2 < \lambda_3$);
  \item[\iiix-block] three constraints concurrent with $\ell_0$ whose
    direction vectors positively span~$\ell_0^{\bot}$;
  \item[\ivp-block] four constraints in hyperboloidal position oriented
    such that only lines lying in the quadric satisfy all four
    constraints;
  \item[\ivx-block] four constraints $\{g_1, \ldots, g_4\}$, where $g_1$
    and $g_2$ lie in a common plane $\Pi$ containing $\ell_0$, $g_3$
    and~$g_4$ meet in a point~$p \in \ell_0$, and $g_1, \ldots, g_4$
    are oriented such that the lines satisfying them are exactly the
    lines in~$\Pi$ passing through~$p$;
  \item[5-block] five constraints oriented such that only~$\ell_{0}$
    satisfies them all.
  \end{densedescription}
\end{lemma}
\begin{proof}
  Let $\FAM$ be a set of $k+1$ constraints, let $N=\{\,\eta_g \mid g \in
  \FAM\,\}$ be the set of their normals, and let $E \subset \R^{4}$ be
  the set of lines satisfying the constraints in~$\FAM$.  Then $N$
  forms a critical $k$-simplex if and only if $E$ is a linear subspace
  of dimension~$4 - k$.

  By construction, this is true for the configurations of constraints
  described in the lemma.  It follows that these configurations are
  indeed blocks.

  It remains to show that if~$\FAM$ is a block, then it falls into one
  of the six cases.  Since every subset of $N$ must be linearly
  independent, we must have $k \leq 4$.  Since $N$ must be linearly
  dependent, we obtain a first necessary condition from
  Lemma~\ref{lem:ortho-dependent}.

  If $k = 1$, then we are in the first case of
  Lemma~\ref{lem:ortho-dependent}: $\FAM$ consists of two
  constraints $\{g_1,g_2\}$ that are either equal or the two
  orientations of the same line $g$. Two equal points cannot form a
  critical segment, so $\FAM$ is a 2-block.

  If $k = 2$, we are in the second case of
  Lemma~\ref{lem:ortho-dependent}: $\FAM$ consists of three
  constraints $\{g_1,g_2,g_3\}$ that are either coplanar or concurrent
  with~$\ell_0$.  If the $g_i$ are coplanar, then $E$ contains the
  lines in the plane $\Pi$ spanned by the~$g_i$.  This set is already
  a two-dimensional linear subspace, and so $E$ must be equal to this
  set.  If two constraints met consecutively by $\ell_0$ have the same
  orientation then one of them is redundant, and $E$ contains other
  lines; and so the orientations must alternate and $\FAM$ is a
  \iiip-block.

  If the $g_i$ are concurrent in some point $p$ on~$\ell_0$, then $E$
  contains all the lines through~$p$.  Again, this set is a
  two-dimensional linear subspace, and therefore equal to~$E$.  Let
  $\Pi$ denote the plane containing the three constraints. A line
  $\ell$ satisfies a constraint $g_i$ if and only if $\ell \cap \Pi$
  lies in the closed halfplane of~$\Pi$ to the right of~$g_i$. The set
  of points of $\Pi$ to the right of all three $g_i$ is reduced to
  $\{p\}$ if and only if the direction vectors of the $g_i$ positively
  span $\ell_0^\bot$.  This is a \iiix-block.

  If $k=3$, we are in the third case of
  Lemma~\ref{lem:ortho-dependent}: $\FAM$ consists either of four
  constraints in hyperboloidal position, or two lines
  concurrent with~$\ell_{0}$ and two other lines coplanar
  with~$\ell_{0}$.

  If the $g_i$ are in hyperboloidal position, then any line in the
  other family of rulings of the quadric containing the~$g_i$
  satisfies all constraints.  This set is a line in~$\R^4$, and so
  must be identical to~$E$.  It follows that $\FAM$ is a \ivp-block.
  Otherwise, $\FAM$ consists of two lines $g_1$ and~$g_2$ coplanar
  with $\ell_0$,
  and two lines $g_3$ and~$g_4$ meeting in a point~$p \in \ell_0$. The
  set $E$ already contains all lines lying in the plane spanned
  by~$g_1 \cup \{p\}$ and passing through~$p$. This set is a line
  in~$\R^{4}$, and therefore identical to~$E$. And so $\FAM$ is a
  \ivx-block.

  Finally, if $k = 4$, then $N$ is a $4$-simplex surrounding the
  origin.  The five constraints $\FAM$ are satisfied only
  by~$\ell_{0}$, and so $\FAM$ is a 5-block.
\end{proof}

\paragraph*{Remark.} Given five orthogonal constraints such that no
four are dependent, we can always orient them (that is, reverse some
of them) so that they pin~$\ell_{0}$, obtaining a 5-block. 


\subsection{Characterization of minimal pinning configurations}
\label{sec:classification-orth}

 Combining the descriptions of Theorem~\ref{thm:surrounding} with the
 characterization of Lemma~\ref{lem:simplices-of-constraints}, we
 obtain the following characterization of minimal pinnings of a line by
 orthogonal constraints.  The numbering of cases corresponds to the
 cases in Figure~\ref{fig:surrounding}.
\begin{theorem}
  \label{thm:char-ortho}
  A family of orthogonal constraints minimally pins $\ell_0$ if and
  only if it forms one of the following configurations:
  \begin{denseitems}
  \item[(1)] A single 5-block;
  \item[(2a)] Two disjoint \iiip-blocks defining distinct planes;
  \item[(2b)] Two disjoint \iiix-blocks meeting $\ell_0$ in distinct
    points;
  \item[(3a)] Two \ivp-blocks sharing two constraints and defining
    distinct quadrics;
  \item[(3b)] Two \ivx-blocks sharing two constraints, such that their
    coplanar pairs define distinct planes or their concurrent pairs
    define distinct points;
  \item[(3c)] A \ivp-block and a \ivx-block sharing two constraints;
  \item[(4a)] A \ivp-block and a \iiip-block sharing one constraint;
  \item[(4b)] A \ivp-block and a \iiix-block sharing one constraint;
  \item[(4c)] A \ivx-block and a \iiip-block sharing one constraint
    such that they define distinct planes;
  \item[(4d)] A \ivx-block and a \iiix-block sharing one constraint
    such that their concurrent pairs meet $\ell_0$ in distinct points;
  \item[(5a)] A \ivp-block and a disjoint 2-block, where the
    2-block constraints are not contained in the quadric defined by
    the \ivp-block;
  \item[(5b)] A \ivx-block and a disjoint 2-block, where the 2-block
    constraints are neither coplanar with the coplanar pair nor
    concurrent with the concurrent pair of the \ivx-block;
  \item[(6a)] A \iiip-block and two 2-blocks, where the
    four 2-block constraints do not all meet, and where no 2-block
    constraint is contained in the plane defined by the \iiip-block;
  \item[(6b)] A \iiix-block and two 2-blocks, where the
    four 2-block constraints do not all meet, and where no 2-block
    constraint goes through the common point of the
    \iiix-block;
  \item[(7)] A \iiip-block and a \iiix-block sharing one constraint,
    and a disjoint 2-block  that does not lie in the plane of the
    \iiip-block and does not go through the common point of the
    \iiix-point;
  \item[(8)] Four disjoint 2-blocks whose supporting lines are not in
    hyperboloidal position (that is, all orientations of four lines
    with finitely many common transversals).
  \end{denseitems}
\end{theorem}
\begin{proof}
  Let $\FAM$ be a family of orthogonal constraints and let $N=\{\,\eta_g
  \mid g \in \FAM\,\}$ be the corresponding family of normals. Recall
  that $\FAM$ minimally pins $\ell_0$ if and only if $N$ minimally
  surrounds the origin in $\R^4$.

  We assume first that $\FAM$ is of one of the 16~types described in
  the theorem.  For each of the types, we can argue that the line
  $\ell_0$ is pinned, since the sets of lines satisfying each of the
  blocks have only $\ell_{0}$ as a common element.  Since $\ell_0$ is
  pinned, $N$ spans $\R^4$.  For all the 16~types, the critical
  simplices of $N$ corresponding to the blocks of~$\FAM$ are of the
  form described in Theorem~\ref{thm:surrounding}, so $N$ minimally
  surrounds the origin, and $\FAM$ minimally pins~$\ell_{0}$.

  It remains to argue the reverse.  We assume that $\FAM$
  minimally pins~$\ell_0$, that is $N$ minimally surrounds the
  origin. Then $N$ is of one of the types described in
  Theorem~\ref{thm:surrounding} and shown in
  Figure~\ref{fig:surrounding}.

  If $|N| = 5$, then $N$ is a critical 4-simplex, and so $\FAM$ is a
  5-block---this is case~(1).  

  If $|N| = 8$, then $N$ is the disjoint union of four critical
  segments, and so $\FAM$ is the disjoint union of four
  2-blocks. Their supporting lines cannot be in hyperboloidal position
  (as then $\FAM$ would not be pinning at all), and so we have
  case~(8).  

  We now consider $|N| = 6$.  If $N$ is the disjoint union of two
  critical triangles, then we must be in case~(2a) or~(2b), as the
  union of a \iiip-block and a \iiix-block cannot pin.  If $N$
  consists of a two critical tetrahedra sharing two constraints, we
  are in cases~(3a)--(3c). If $N$ consists of a critical tetrahedron
  and a critical triangle sharing one constraint, we are in
  cases~(4a)--(4d).  If $N$ consists of critical tetrahedron and a
  disjoint critical segment, we have cases~(5a)--(5b).  The additional
  conditions in cases~(2a)--(3b) and~(4c)--(4d) hold as otherwise
  $\FAM$ would not be pinning at all.

  We finally turn to the case $|N| = 7$. Here we need an additional
  observation: If two \iiip-blocks share a constraint, then they both
  define the same plane.  That implies that the critical triangles
  formed by the normals span the same 2-space, and so these triangles
  cannot appear together in a point set minimally surrounding the
  origin by Lemma~\ref{lem:overlap}.  It follows that $\FAM$ cannot
  contain two \iiip-blocks sharing a constraint, and the same
  reasoning applies to two \iiix-blocks.  Therefore $\FAM$ cannot
  contain three 3-blocks sharing a common constraint, and if it
  contains two 3-blocks sharing a constraint, then these must be a
  \iiip-block and a \iiix-block. This implies that we must be in
  cases~(6a),~(6b) or~(7).  The additional conditions again hold as
  otherwise~$\FAM$ would not be pinning.
\end{proof}


\section{Minimal pinning configurations by general constraints}
\label{sec:minpin}

We now return to constraints that are not necessarily orthogonal. As
we saw before, the boundary $\FU_{g}(u) = 0$ of the set $\VOL_{g}$ is
a quadric through the origin, with outward normal $\NORM_{g}$~in the
origin.  Interestingly, this normal $\NORM_{g}$ only depends on
$\lambda$ and $\alpha$, but not on the slope~$\delta$.  If we consider a
constraint $g(\lambda, \alpha, \delta)$ for varying $\delta$, the
shape of $\partial \VOL_{g}$ changes, but its outward
normal in the origin remains the same.  In particular, when $\delta$
reaches zero, the quadratic term in $\FU_{g}$ vanishes, and we have
$\FU_{g}(u) = \NORM_{g} \cdot u$. In other words, linearizing the
volume $\VOL_{g}$ corresponds to replacing~$g$ by its projection
$\gORTH$ on the plane perpendicular to $\ell_0$ in $\ell_0 \cap g$;
$\VOL_{\gORTH}$ is the halfspace with outer normal $\NORM_{g}$ through
the origin. We call $\gORTH$ the \emph{orthogonalized constraint} of
$g$, and denote by $\FAMO$ the family of orthogonalized constraints of
$\FAM$. (Note that $\FAMO$ can have smaller cardinality than~$\FAM$.)

\paragraph{First order pinning.}
As in the previous section, we call a family of constraints~$\FAM$
\emph{dependent} if their normals are linearly dependent.  Since the
observation above shows that the normal to a constraint~$g$ is
identical to the normal of~$\gORTH$, $F$ is dependent if and only if
$\FAMO$ is dependent.  For a characterization of dependent families,
we can simply refer to Lemma~\ref{lem:ortho-dependent}, but that
leaves a question: What does it mean geometrically for four
constraints if their orthogonalized constraints are in hyperboloidal
position? The following lemma clarifies this. Given three pairwise
skew lines $g_{1}, g_{2}, g_{3}$, we let $\BOLOID(g_1, g_2, g_3)$ be
the quadric formed by the transversals to the three lines.
\begin{lemma}
  \label{lem:degeneracies}
  Two constraints are dependent if and only if they are at the same
  time coplanar and concurrent with~$\ell_{0}$.
  
  Three constraints are dependent if and only if (i) two of them are,
  or if (ii) the constraints are coplanar with~$\ell_{0}$, or if (iii)
  the constraints are concurrent with~$\ell_{0}$.

  Four constraints $g_{1}, g_2, g_3, g_4$ are linearly dependent if
  and only if (i) three of them are, or if (ii) $g_4$ is tangent to
  the quadric $\BOLOID(g_1, g_2, g_3)$ at~$\ell_{0} \cap g_{4}$, or if
  (iii) two of the lines are concurrent with~$\ell_{0}$ and the other
  two are coplanar with~$\ell_{0}$.
\end{lemma}
\begin{proof}
  The normals of constraints are linearly dependent if and only if the
  normals of their orthogonalized constraints are.  This immediately
  implies the statements for two and three constraints, and nearly
  implies the statement for four constraints.  It remains to show that
  $\gORTH_1, \dots, \gORTH_4$ are in hyperboloidal position if and
  only if $g_{4}$ is tangent to the quadric $\BOLOID(g_1, g_2, g_3)$
  at~$\ell_{0} \cap g_{4}$.
  
  We can assume that the lines $g_{1}, g_{2}, g_{3}$ are pairwise
  skew, so the quadric $\BOLOID = \BOLOID(g_{1}, g_{2}, g_{3})$ is
  defined.  Since $\ell_0$ is contained in $\BOLOID$, for any point $p
  \in \ell_0$ there is a unique constraint contained in~$\BOLOID$ and
  passing through~$p$. In particular, let $g^{*}=g(\lambda_4,
  \alpha^{*}, \delta^{*})$ denote the constraint contained in
  $\BOLOID$ through $p = \ell_0 \cap g_4 = (0,0,\lambda_4)$. Since the
  curve $\bigcap_{i=1}^{3} \partial \VOL_{g_i}$ is contained in
  $\partial \VOL_{g^{*}}$, its tangent in the origin, which is
  orthogonal to the normals $\NORM_{g_1}, \ldots, \NORM_{g_3}$, is
  contained in the tangent plane to $\VOL_{g^{*}}$ at the origin.
  This implies that $\NORM_{g^{*}}$ lies in the linear subspace~$E$
  spanned by~$\NORM_{g_1}, \ldots, \NORM_{g_3}$.  Let ${w} \in \R^{4}$
  be such that $E = \{\,{w} \cdot {u} = 0 \mid {u} \in \R^4\,\}$.
  Consider another constraint $g = g(\lambda_{4}, \alpha, \delta)$
  through~$p$. The normal $\NORM_{g}$ lies in~$E$ if and only if
  \[
  0 = w \cdot \NORM_{g} = w \cdot \NORM(\lambda_{4}, \alpha) = 
  \left(\begin{array}{c}
    w_{1} (1-\lambda_{4}) + w_{3}\lambda_{4}\\
    -w_{2} (1-\lambda_{4}) - w_{4}\lambda_{4}
  \end{array}\right) 
  \cdot
  \left(\begin{array}{c}
    \sin \alpha \\
    \cos \alpha
  \end{array}\right),
  \]
  that is when $(\sin \alpha, \cos \alpha)$ is orthogonal to a fixed
  vector.  Modulo~$\pi$, that is, up to reversal of the line, there is
  only one solution for~$\alpha$, namely when $\alpha = \alpha^{*}$.

  It follows that the normals of $g_{1}, g_{2}, g_{3}$ and $g$ are
  linearly dependent if and only if $g$ is coplanar and concurrent
  with~$g^{*}$ and~$\ell_0$.  This is true if and only if $g$ is tangent
  to~$\BOLOID$ in~$p$.
\end{proof}

By approximating the volumes $\VOL_g$ to first order, we get:
\begin{lemma}
  \label{lem:generic-linearization}
  Let $\FAM$ be a family of constraints. If $\FAMO$ pins~$\ell_0$ then
  $\FAM$ pins~$\ell_0$.  If $\FAM$ pins $\ell_0$ and no four
  constraints in $\FAM$ are dependent, then $\FAMO$ pins~$\ell_0$.
\end{lemma}
\begin{proof}
  Since the sets $\VOL_{g}$ are bounded by algebraic surfaces of
  constant degree, the origin $0$ is isolated in the intersection of
  such volumes if and only if there exists no smooth path moving away
  from~$0$ inside that intersection. If the tangent vector at~$0$ to a
  smooth path~$\gamma$ makes a positive dot product with $\NORM_g$,
  then $\gamma$ locally exits $\VOL_g$. If $\FAMO$ pins $\ell_0$, then
  the normals to $\FAMO$ surround the origin and any vector must make
  a positive dot product with the normal to at least one of the
  constraints in~$\FAMO$. Since these are also the normals of the
  constraints in~$\FAM$, no smooth path can move away from $0$ inside
  $\bigcap_{g \in \FAM}\VOL_g$, and $\FAM$ also pins~$\ell_0$. The
  same argument shows that if $\FAM$ pins~$\ell_0$, then
  $\bigcap_{\gORTH \in \FAMO}\VOL_{\gORTH}$ must have empty
  interior. In that case, if it is not a single point then four of the
  normals to the constraints in~$\FAMO$ are linearly dependent, and
  the statement follows.
\end{proof}
Lemma~\ref{lem:generic-linearization} implies that the minimal pinning
examples we gave in the previous section are surprisingly robust: We
can start with a family of orthogonal constraints pinning~$\ell_{0}$,
rotate each constraint by changing its $\delta$-parameter arbitrarily,
and the resulting family will still pin~$\ell_{0}$. However, the lemma
does not exclude the possibility that there are pinnings that are not
robust in this sense, and indeed this is the case. We saw in the
previous section that minimal pinnings by orthogonal constraints
consist of between five and eight lines.  Surprisingly, it is possible
to pin using only \emph{four} non-orthogonal constraints, as we will
see below.

Our Isolation Lemma analyzes the intersection of the cone
$C=\bigcap_{g\in \FAM}\VOLV_g$ with $\MANI$ near $0$ in terms of the
trace of $C$ on the hyperplane $T$. Perhaps surprisingly, this trace
has a simple geometric interpretation. Consider the map $\phi\colon \R^{4}
\rightarrow \R^{5}$ with $\phi(u_{1},u_2,u_3,u_4) =
(u_1,u_2,u_3,u_4,0)$. Clearly, $\phi$ defines a bijection
between~$\LINES$ and the hyperplane~$T = \{u_5 = 0\}$. Moreover, given
a constraint $g$, a line $\ell(u)$ satisfies $\gORTH$ if and only if
$\phi(u) \in \VOLV_g$. Thus, $\phi^{-1}(C \cap T)$ represents, in our
$\R^4$ parameterization of $\LINES$, those lines that satisfy $\FAMO$.
In particular, $\FAMO$ pins~$\ell_{0}$ if and only if $C \cap T =
\{0\}$.  In the proof of Lemma~\ref{lem:general-eight}, we observed
that if $\FAM$ minimally pins $\ell_0$ and $C \cap T \neq \{0\}$ then
$\FAM$ has size at most six. This implies:
\begin{lemma}
  \label{lem:general-six}
  If $\FAM$ is a minimal set of constraints pinning~$\ell_{0}$ such
  that $\FAMO$ does not pin~$\ell_0$ then $\FAM$ has size at most six.
\qed
\end{lemma}

\paragraph{Examples of higher-order pinnings.}
In the proof of Lemma~\ref{lem:general-eight}, we had considered two
cases, namely where $C \cap T = \{0\}$ (this includes the case where
$C$ is a line intersecting~$\MANI$ transversely), and where $E = \ah{C
  \cap T}$ is a $k$-space with $1 \leq k \leq 3$.  In the former case
only~$\ell_{0}$ satisfies~$\FAMO$, so $\FAMO$ pins~$\ell_{0}$.  In the
latter case, however, all lines in $\phi^{-1}(C\cap T)$
satisfy~$\FAMO$, and so~$\FAMO$ does not pin.  We give a few such
examples, also showing that dimensions~1, 2,~and~3 are all possible
for~$E = \ah{C\cap T}$.

First, consider the four constraints 
\begin{align*}
  g_0 &= \{\,(-t, 0, 0) \mid t \in \R\,\},\\
  g_1 & = \{\,(t, t, 1)\mid t \in\R\,\},\\
  g_2 &= \{\,(-t, -2t, 2)\mid t \in \R\,\}, \hbox{ and }\\
  g_3 &= \{\,(t, 3t, 3)\mid t \in \R\,\},
\end{align*}
oriented in the direction of increasing~$t$.  The four constraints lie
in the quadric~$\BOLOID$ defined by the equation~$y = xz$ and have
alternate orientations, so they form a \ivp-block.  A line $\ell$
satisfies the four constraints if and only if it lies in the other
family of rulings of~$\BOLOID$.  These lines are of the
form~$\{\,(a,at,t)\mid t \in \R\,\}$, for $a \in \R$.  We now replace
$g_0$ by $g'_0$, by rotating the constraint slightly around the origin
inside the plane $y = 0$ (the tangent plane to~$\BOLOID$ in the
origin).  For concreteness, we choose $g'_{0}$ to be the line~$\{(-t,
0, -t/100)\}$.  This means that for points on $g'_0$, we have $y - xz
= 0 - t^2/100 < 0$ for $t \neq 0$. This implies that $g'_{0}$
touches~$\BOLOID$ in the origin, and otherwise lies entirely in the
volume $y < xz$ bounded by~$\BOLOID$. In order to satisfy the four
constraints $g'_{0}, g_1, g_2, g_{3}$, a line near $\ell_{0}$ would
have to intersect~$\BOLOID$ at least three times (points of tangency
counted twice), and thus lies in~$\BOLOID$ as $\BOLOID$ is a
quadric. A line close to~$\ell_0$, but distinct from it, that is
contained in $\BOLOID$ must violate~$g'_0$. It follows that the family
$\FAM=\{g'_{0}, g_{1}, g_{2}, g_3\}$ pins~$\ell_{0}$. Since $\FAMO =
\{g_{0}, g_{1}, g_{2}, g_{3}\}$ does not pin~$\ell_{0}$, this example
already shows that the condition of independence in the second
statement of Lemma~\ref{lem:generic-linearization} is necessary. Note
that in this example the space $E = \ah{C \cap T}$ is one-dimensional
(it is the set of transversals to~$\FAMO$).

Second, consider the family~$\FAM$ consisting of the six lines $g_1 =
g(0, -\pi/2, -1)$, $g_2 = g(0, \pi/2, -1)$, $g_3 = g(1, 0, 1)$, $g_4 =
g(1, \pi, 1)$, $g_5 = g(2, 3\pi/4, 0)$, and $g_6 = g(2, -\pi/4, 0)$.
The family~$\FAMO$ consists of three 2-blocks, that is all
orientations of three lines.  The set of lines satisfying~$\FAMO$ is
the set of transversals to these three lines, again a one-dimensional
family.  We can verify that $\FAM$ pins~$\ell_{0}$ using the Isolation
Lemma.  The halfspaces corresponding to the lines are:
\begin{align*}
  \VOLV_{g_1}\colon & -u_1 - u_5 \leq 0,\\
  \VOLV_{g_2}\colon & \hphantom{{}-{}}u_1 - u_5 \leq 0,\\
  \VOLV_{g_3}\colon & -u_4 + u_5 \leq 0,\\
  \VOLV_{g_4}\colon & \hphantom{{}-{}}u_4 + u_5 \leq 0,\\
  \VOLV_{g_5}\colon & -u_1 - u_2 + 2u_3 + 2u_4 \leq 0, \text{ and}\\
  \VOLV_{g_6}\colon & \hphantom{{}-{}}u_1 + u_2 - 2u_3 - 2u_4 \leq 0.
\end{align*}
The constraints $g_1$ and $g_2$ imply $u_5 \geq 0$, and
constraints $g_3$ and $g_4$ imply $u_5
\leq 0$. 
 Together they enforce $u_5=0$ and then
$u_1 = u_4 = 0$.  Plugging this into the last two constraints we
obtain~$u_2 = 2 u_3$.  Since the 1-space $\{\,(0, 2t, t, 0, 0)\mid t \in
\R\,\}$ intersects~$\MANI$ in the origin only, $\FAM$ pins~$\ell_{0}$.
None of the constraints is redundant, as can be checked by showing
that for each $g_i$ there is a point in $T \cap \MANI$ satisfying all
but this constraint.

Third, replace the lines $g_5$ and~$g_6$ in the previous family by
$g'_5 = g(2, -3\pi/4, 0)$ and $g'_6 = g(3, \pi/4, 0)$. This produces
two different halfspaces
\begin{align*}
  \VOLV_{g'_5}\colon & \hphantom{{}-{}} u_1 - u_2 - 2u_3 + 2u_4 \leq
  0, \text{ and}\\
  \VOLV_{g'_6}\colon & -2u_1 + 2u_2 + 3u_3 - 3u_4 \leq 0.
\end{align*}
Plugging $u_1 = u_4 = u_5 = 0$ into these constraints we obtain
$-u_2/2 \leq u_3 \leq -2u_2/3$.  This is a two-dimensional wedge in the
$(u_2, u_3)$-plane.  It follows that $E = \ah{C \cap T}$ is the
2-space $u_1 = u_4 = u_5 = 0$. The intersection $E\cap \MANI$ consists
of the two 1-spaces $u_2 = 0$ and $u_3 = 0$, which intersect the wedge
$-u_2/2 \leq u_3 \leq -2u_2/3$ in the origin only.  It follows that
$\FAM = \{g_1, g_2, g_3, g_4, g'_5, g'_6\}$ pins~$\ell_{0}$.  Again,
we can check minimality by verifying that no constraint is redundant.

Finally, we consider the family $\FAM =\{g_1, g_2, g''_3, g''_4,
g''_5, g''_6\}$, where $g''_3 = g(-1, \pi - \tau_3, 0)$, $g''_4 =
g(-1/2, -\tau_2, 0)$, $g''_5 = g(1/4, \tau_2, 0)$, and $g''_6 = g(1/3,
\pi + \tau_3, 0)$, with $\tau_2 = \arctan 2 \approx 63.4^{\circ}$ and
$\tau_3 = \arctan 3 \approx 71.6^{\circ}$.  The corresponding
halfspaces are
\begin{align*}
  \VOLV_{g_1}\colon & -u_1 - u_5 \leq 0,\\
  \VOLV_{g_2}\colon & \hphantom{{}-{}} u_1 - u_5 \leq 0,\\
  \VOLV_{g''_3}\colon & \hphantom{{}-{}} 6 u_1 + 2 u_2 - 3 u_3 - u_4 \leq 0,\\ 
  \VOLV_{g''_4}\colon & -6 u_1 - 3 u_2 + 2 u_3 + u_4 \leq 0,\\
  \VOLV_{g''_5}\colon & \hphantom{{}-{}}  6 u_1 - 3 u_2 + 2 u_3 - u_4
  \leq 0, \text{ and}\\
  \VOLV_{g''_6}\colon & -6 u_1 + 2 u_2 - 3 u_3 + u_4 \leq 0
\end{align*}
The first two constraints again ensure $u_5 \geq 0$.
To construct $C\cap T$, we note that
$u_5 = 0$ implies $u_1 = 0$.  Plugging $u_1 = u_5 = 0$ into the
remaining four constraints, we obtain a three-dimensional cone with apex
at the origin that lies in the octant $u_2, u_3, u_4 > 0$.  It follows
that $E = \ah{C \cap T}$ is the 3-space $u_1 = u_5 = 0$.  The
intersection $E \cap \MANI$ consists of the two 2-spaces $u_2 = 0$ and
$u_3 = 0$, both of which intersect~$C$ only in the origin.  It follows
that $\FAM$ pins~$\ell_{0}$, and minimality is verified by checking
that no constraint is redundant.

\paragraph{The unique nondegenerate minimal pinning of higher order.}
Call a pair of constraints~$\{g_1, g_2\}$ \emph{degenerate} if they
are at the same time coplanar and concurrent with~$\ell_{0}$.  This is
equivalent to $\gORTH_1 = \gORTH_2$ or $\{\gORTH_1, \gORTH_2\}$
forming a 2-block. We do not attempt to give a full characterization
of all possible pinning configurations~$\FAM$ where $\FAMO$ does not
pin.  However, we observe that all but one of the examples given above
contained degenerate pairs of constraints. We can show that the
example we gave is in fact unique in this respect.  

\begin{theorem}
  \label{thm:4-pinning}
  Let $\FAM$ be a minimal set of constraints pinning~$\ell_{0}$ not
  containing degenerate pairs and such that $\FAMO$ does not
  pin~$\ell_{0}$.  Then $\FAM$ consists of exactly four constraints
  and~$\FAMO$ is a \ivp-block.
\end{theorem}
\begin{proof}
  Let $C = \bigcap_{g \in \FAM}\VOLV_g$. Since $\FAMO$ does not pin,
  we have $C \cap T \neq \{0\}$.  Let $E = \ah{C\cap T}$, and define
  $\G = \{\,g \in \FAM \mid E \subset \VOLV_g\,\}$.  By
  Lemma~\ref{lem:cone-decomposition} applied to the cone~$C \cap T$,
  we have $\bigcap_{g \in \G} \VOLV_g \cap T = E$. Since
  $\phi^{-1}(\VOLV_g \cap T) = \VOL_{\gORTH}$, the set $\pE =
  \phi^{-1}(E)$ represents, in our parameterization of $\LINES$ by
  $\R^4$, the set of lines satisfying~$\GO$. In other words, $\GO$ is
  a family of orthogonal constraints such that the set~$\pE$ of lines
  satisfying~$\GO$ is a $k$-dimensional subspace, where $k \in
  \{1,2,3\}$.

  We now argue that~$k=1$ and that~$\pE$ is the set of lines
  satisfying a \ivp-block. Let $N$ denote the set of normals to the
  constraints in~$\GO$, and observe that $N$ surrounds the origin in
  the $(4-k)$-dimensional subspace orthogonal to $\pE$. Since $\FAM$
  contains no degenerate pair, $\GO$ contains no 2-block, and so $N$
  contains no critical segment.  This immediately implies~$k \neq
  3$. If $k = 2$, then $N$ surrounds the origin in a 2-space.  Since
  $N$ contains no critical segment, it must contain a critical
  triangle, and by Lemma~\ref{lem:simplices-of-constraints} $\pE$ is
  the set of lines satisfying a 3-block. $\pE$~is thus either the set
  of lines lying in fixed plane~$\Pi\supset \ell_{0}$, or the set of
  lines through a fixed point~$p \in \ell_{0}$.  In both cases, all
  such lines meet~$\ell_{0}$. However, if $\ell(u)$ meets $\ell_0$
  then $\psi(u) \in T$ and $\psi(u) = \phi(u)$. It follows that $C
  \cap T \subset E \subset \MANI$, and $\FAM$ cannot pin~$\ell_{0}$, a
  contradiction. It follows that $k = 1$, so $\pE$ is one-dimensional
  and $N$ surrounds the origin in the 3-space orthogonal to~$\pE$.
  Since $N$ contains no critical segment, there are two cases by
  Lemma~\ref{lem:surrounding-3d}: First, $\pE$ could be the set of
  lines satisfying two 3-blocks sharing one constraint, which is the
  set of lines lying in a fixed plane and going through a fixed point.
  All such lines meet~$\ell_{0}$, a contradiction.  Second, $\pE$
  could be the set of lines satisfying a 4-block.  If this is a
  \ivx-block, then again all lines in~$\pE$ meet~$\ell_{0}$, a
  contradiction.  We have thus established that $\pE$ is the set of
  lines satisfying a \ivp-block, that is one family of rulings of a
  quadric~$\BOLOID$, while $\GO$ is a subset of the other family of
  rulings of~$\BOLOID$.

  To conclude, we will now show that $\FAM$ has size four. Since
  $\FAM$ pins~$\ell_0$ and $C \cap T \neq \{0\}$, we are in case~(ii)
  or~(iii) of the Isolation Lemma.  Without loss of generality, we
  assume case~(ii), and have $C \subset T^{>}\cup \{ \MANI^{>} \cap
  T\} \cup \{0\}$.  This implies $C \cap T \subset \MANI^{>} \cup
  \{0\}$, and by Lemma~\ref{lem:homogeneous}~(ii) we have $E \subset
  \MANI^{>} \cup \{0\}$. By Lemma~\ref{lem:cone-decomposition2} we
  have $\bigcap_{g \in \G} \VOLV_g \subset T^{\geq}$, and since
  $\bigcap_{g \in \G} \VOLV_g \cap T = E \subset \MANI^{>} \cup
  \{0\}$, this implies $\bigcap_{g \in \G} \VOLV_{g} \subset
  T^{>} \cup \{\MANI^{>} \cap T\} \cup \{0\}$.  By
  Lemma~\ref{lem:isolation}, the family $\G$ pins~$\ell_{0}$, and by
  minimality of $\FAM$ we have $\FAM = \G$, and so the constraints
  in~$\FAMO$ are in hyperboloidal position.

  Let $Y = \{u_{5} \leq -1\}$.  Since $\bigcap_{g \in \FAM} \VOLV_{g}
  \subset T^{\geq}$, we have $\big(\bigcap_{g \in \FAM} \VOLV_{g}\big)
  \cap Y = \emptyset$.  Since $E \subset \partial \VOLV_g$ for~$g \in
  \FAM$ and $E$ is parallel to~$\partial Y$, we can project all
  halfspaces on the 4-space orthogonal to~$E$ and apply Helly's
  theorem there to obtain a five-element subset of $\{\,\VOLV_g\mid g
  \in \FAM\,\} \cup \{Y\}$ with empty intersection. Since $Y$ must be
  one of the five elements, there is a four-element subset $\FAM_1
  \subset \FAM$ with $\bigcap_{g \in \FAM_1} \VOLV_{g} \subset
  T^{\geq}$.  Consider the cone $C_1 = \bigcap_{g \in \FAM_1} \VOLV_g$
  and let $E_1 = \ah{C_1 \cap T}$.  If $E_1$ is one-dimensional, then $E
  = E_1$, and $\FAM_1$ pins~$\ell_{0}$, implying that $\FAM = \FAM_1$.
  Since $\FAMO$ is in hyperboloidal position, it is a \ivp-block. It
  remains to consider the case that $E_1$ is at least two-dimensional.
  Let $\FAM_2 = \{\, g\in \FAM_1 \mid E_1 \subset \VOLV_g\,\}$.  By
  Lemma~\ref{lem:cone-decomposition}, we have $\bigcap_{g\in
    \FAM_2}\VOLV_g \cap T = E_1$.  The normals of the constraints in
  $\FAMO_2$ would have to surround the origin in a 1-space or 2-space,
  but this is impossible for constraints in hyperboloidal position
  without a 2-block.
\end{proof}

\section{Proofs of Theorems~\ref{thm:finite},~\ref{thm:finite-lines},
  and~\ref{thm:infinite}} 
\label{sec:wrapup}

We can now finish the proof of Theorem~\ref{thm:finite-lines},
which states that any minimal pinning of a line by constraints has
size at most eight. The bound reduces to six if no two constraints are
simultaneously concurrent and coplanar with $\ell_0$, that is form a
degenerate pair.

\begin{proof}[Proof of Theorem~\ref{thm:finite-lines}]
 The first statement was proven in Lemma~\ref{lem:general-eight}. For
 the second statement, consider be a minimal pinning $\FAM$ of
 $\ell_0$ by constraints, no two forming a degenerate pair. If $\FAMO$
 does not pin $\ell_0$ then by Lemma~\ref{lem:general-eight} we have
 that $\FAM$ has size at most six. If $\FAMO$ pins~$\ell_{0}$ but is
 not minimal, then some subfamily $\G \subsetneq \FAM$ is such that
 $\GO$ pins~$\ell_{0}$; Lemma~\ref{lem:generic-linearization} yields
 that $\G$ pins~$\ell_{0}$, a contradiction. If $\FAMO$ is a minimal
 pinning of $\ell_0$, since it cannot contain a $2$-block (otherwise
 $\FAM$ would contain a degenerate pair) it must have size at most six
 by Theorem~\ref{thm:char-ortho}.
\end{proof}

We now return to families of convex polytopes that pin a line, and
prove Theorem~\ref{thm:finite}, which asserts that any minimal pinning
of a line by polytopes in $\R^3$ has size at most eight if no facet of
a polytope is coplanar with the line.  The number reduces to six if,
in addition, the polytopes are pairwise disjoint.

\begin{proof}[Proof of Theorem~\ref{thm:finite}]
  Consider a family~$\FAM$ of convex polytopes pinning the
  line~$\ell_{0}$, such that no facet of a polytope is coplanar
  with~$\ell_{0}$.  Then, for each polytope $F \in \FAM$, $\ell_{0}$
  intersects $F$ either in the interior of an edge~$e^F$ or in a
  vertex~$v$.  In the former case, a line $\ell$ in a neighborhood
  of~$\ell_{0}$ intersects~$F$ if and only if it satisfies a
  constraint supporting~$e^{F}$. In the latter case, there are exactly
  two silhouette edges $e^{F}_{1}$ and~$e^{F}_{2}$ incident to~$v$ in
  the direction of~$\ell_{0}$, and a line $\ell$ in a neighborhood of
  $\ell_{0}$ intersects~$F$ if and only if it satisfies the two
  constraints supporting $e^{F}_{1}$ and~$e^{F}_{2}$.  It follows that
  the family of these constraints pins~$\ell_{0}$, and so
  Theorem~\ref{thm:finite-lines} implies that $\ell_{0}$ is already
  pinned by eight of the constraints.  The corresponding at most eight
  polytopes pin~$\ell_{0}$.  If we now make the additional assumption
  that the polytopes are pairwise disjoint, then no two constraints
  can be coplanar and concurrent with $\ell_0$. It then follows from
  the second statement in Theorem~\ref{thm:finite-lines} that six
  constraints suffice to pin the line, and the statement for pairwise
  disjoint polytopes follows.
\end{proof}

Last, we give a construction of arbitrarily large minimal pinnings of
a line by polytopes in $\R^3$, proving Theorem~\ref{thm:infinite}.

\begin{proof}[Proof of Theorem~\ref{thm:infinite}]
  \begin{figure}
    \centerline{\includegraphics{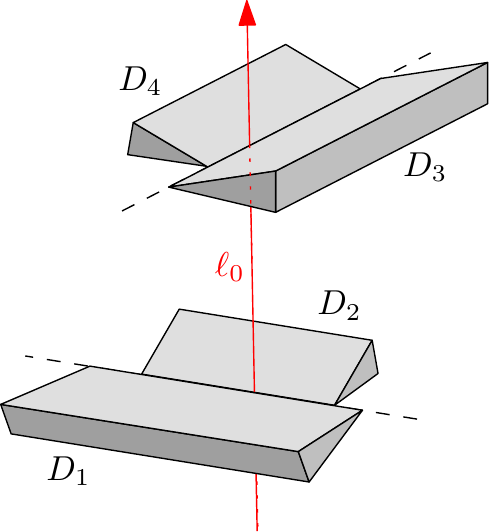}}
    \caption{Four polytopes that restrict transversals to a
      two-dimensional set.}
    \label{fig:lower-bound}
  \end{figure}
  We again identify $\ell_0$ with the $z$-axis. We first pick two
  polytopes $D_1$ and $D_2$ such that their common transversals in the
  vicinity of $\ell_0$ are precisely the lines intersecting the
  $y$-axis. Similarly, we pick two polytopes $D_3$ and $D_4$ that
  restrict the transversals to pass through the line $\{\,(t, 0, 1)\mid
  t \in \R\,\}$, as in Figure~\ref{fig:lower-bound}.  A line $\ell(u)$
  meets all four polytopes if and only if $u_1 = u_4 = 0$.  We can
  therefore analyze the situation in the $u_2u_3$-plane.  We add two
  other polytopes $D_5$ and $D_6$ (not pictured) to enforce $2 u_2
  \geq u_3$ and $2 u_3 \geq u_2$; these polytopes are bounded by the
  oriented lines $\{\,(t,-t,-1)\mid t \in \R\,\}$ and~$\{\,(t,-t,2)\mid t
  \in \R\,\}$.  
  In the $u_2u_3$-plane, the
  set of lines meeting $D_1,\dots,D_6$ is the closed wedge $W =
  \{\,(u_2, u_3) \mid u_2/2 \leq u_3 \leq 2u_2\,\}$.

  Consider two angles, $\beta$ and~$\theta$, with $0 <
  \beta < \theta < \pi/2$.  Let $v=(v_x,v_y)=(\cos \beta, \sin
  \beta)$, $w=(w_x,w_y) = (\cos \theta, \sin \theta)$ be two unit
  vectors, and define the unbounded polyhedral wedge $F(v,w) =
  \{\,(x,y,z)\mid v_x x + v_y y \leq 0 \text{ and } w_x x + w_y y \leq
  0\,\}$.  The left-hand side of Figure~\ref{fig:counter} shows a
  projection along~$\ell_{0}$.  A line $\ell(u)$ with $u_2, u_3 > 0$
  and $u_1 = u_4 = 0$ misses $F(v,w)$ if and only if the vector $(u_2,
  u_3)\in\R^2$ falls in the (closed, counterclockwise) acute angular interval
  $\xi_{v,w}=[\beta,\theta]$ between $v$ and~$w$.  In other words, the
  set of lines intersecting~$F(v,w)$ looks like the gray shape in the
  $u_2u_3$-plane depicted on the right hand side of
  Figure~\ref{fig:counter}, namely the plane with the closed wedge
  corresponding to~$\xi_{v,w}$ removed.

   We pick $n$~pairs of vectors $(v^1,w^1)$, $(v^2,w^2)$, \dots,
   $(v^n,w^n)$, with the property that together the wedges cover~$W$,
   and such that the middle vector $v^i + w^i$ of $\xi_{v^{i}w^{i}}$
   lies in~$W$ but does not lie in $\xi_{v^{j}w^{j}}$ for~$j \neq i$.
   The family $\{F_1, F_2, \dots, F_n\}$ of shapes $F_i=F(v^i,w^i)$
   has the property that $\ell_0$ is the only line in~$W$ intersecting
   the entire family, but for any $1 \leq i \leq n$ there is an entire
   sector of such lines that intersect all~$F_{j}$ with~$j \neq i$.
   It follows that the family $\FAM = \{D_1, \dots, D_6, F_1, \dots,
   F_n\}$ pins $\ell_{0}$, but has no pinning subfamily of size
   smaller than~$n$ (some $D_i$ could be redundant, but none of the
   $F_i$ is).

   In the final step, we crop the polyhedral wedges to create a family
   of bounded polytopes with the same property.  Since pinning is
   determined by lines in a neighborhood of~$\ell_{0}$ only, we can
   clearly make $D_1, \dots, D_6$ bounded.  For each $F_{i}$, select
   the following ``linear'' path $\gamma_{i}(t)$ in $\LINES$ starting
   at the origin: $\gamma_{i}(t)= (u_1(t), u_2(t),u_3(t),u_4(t)) = (0,
   v_y^i+w_y^i,v_x^i+w_x^i, 0)\cdot t$, for $0\leq t\leq 1$. In the
   projection on the $xy$-plane, the line $\gamma_{i}(t)$ is, for any
   $t > 0$, perpendicular to the vector~$v^i+w^i$.  Therefore
   $\gamma_{i}(t)$ does not intersect~$F_i$, but by construction it
   intersects each $F_{j}$ with $j\neq i$. Let $P_{ij}(t)$ be the
   point where $\gamma_{i}(t)$ enters or exits~$F_{j}$. A
   straightforward calculation shows that $P_{ij}(t)$ moves linearly
   away from $\ell_0$ along a line perpendicular to $\ell_0$.

   \begin{figure}[tb]
     \centerline{\includegraphics{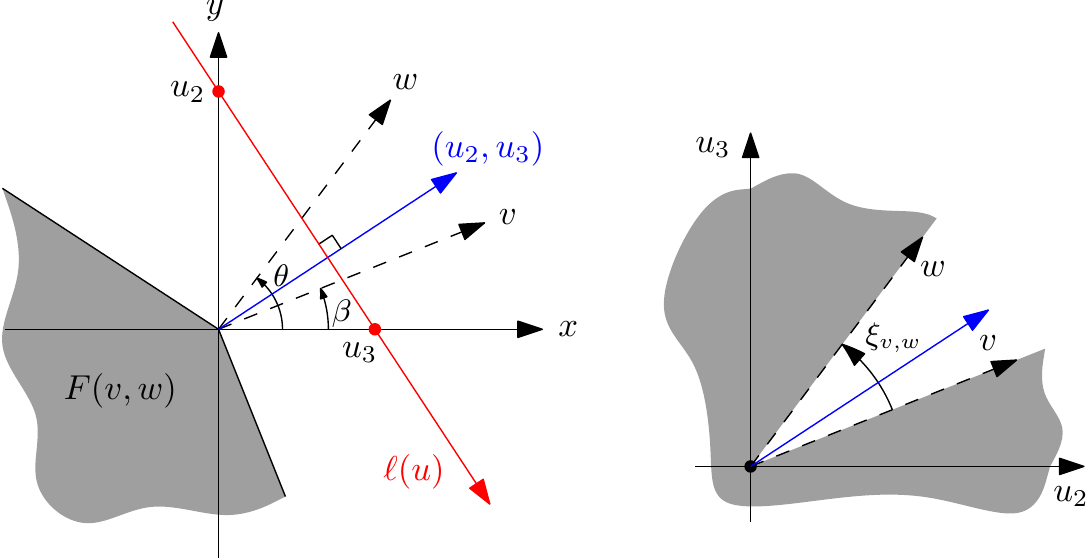}}
     \caption{The polyhedron~$F(v,w)$ and the set of lines
       intersecting it.}
     \label{fig:counter}
   \end{figure}

   We can therefore crop each $F_{j}$ to a bounded polytope, ensuring
   that it still contains all points $P_{ij}(t)$ for $i\neq j$ and
   $0< t\leq 1$, and hence intersects all lines on the paths
   $\gamma_{i}(t)$ for~$i \neq j$.  As a result, the family $\FAM$
   still pins~$\ell_{0}$, but for any $F_{i}$ the family $\FAM
   \setminus\{F_i\}$ is not a pinning as witnessed by the
   path~$\gamma_{i}$.
 \end{proof}

\section{Concluding remarks}

We have shown that minimal pinnings by convex polytopes have bounded
size if the line~$\ell_{0}$ is not coplanar with a facet of a
polytope.  It seems that this condition can be slightly relaxed: if
$\ell_{0}$ intersects a polytope in a vertex~$v$ and lies in the plane
of a facet incident to~$v$, then one can argue separately
in~$T^{\geq}$ and~$T^{\leq}$.  However, we do not know if the result
generalizes any further.

If $\ell_0$ meets the relative interiors of two different edges of a
polytope facet (and is thus coplanar with this facet), then a line
$\ell$ near $\ell_{0}$ intersects the polytope if and only if it
satisfies one of the two line constraints that support the front edge
and the back edge of the facet.  It follows that the set of lines
intersecting the polytope is the union of two ``halfspaces.''  Since
this is not a convex shape, our techniques do not seem to apply.

We have shown that intersecting convex polytopes can have minimal
pinnings of arbitrary size.  We conjecture that \emph{pairwise
disjoint} convex polytopes have bounded pinning number.  In fact, as
mentioned in the introduction, we are not aware of any construction of
a minimal pinning by
\emph{arbitrary} pairwise disjoint compact convex objects of size
larger than six.

Lemma~\ref{lem:general-eight} is a key lemma.  The most difficult case is
when $E$ is three-dimensional, and in fact the proof only works because
$E\cap \MANI$ must consist of two intersecting 2-spaces.  The proof
would not go through if $E\cap\MANI$ were allowed to be, for example,
a circular cone.  Perhaps this is an indication that 
there might be no corresponding result
 in higher dimensions.

We saw in the introduction that a family of lines pinning a line can
be considered as a grasp of that line.  In grasping, one often
considers \emph{form closure}, which means that the object is
immobilized even with respect to infinitesimally small movements.  For
instance, an equilateral triangle with a point finger at the midpoint
of every edge is immobilized, as it cannot be moved in any way, but it
is not in form closure because an infinitesimal rotation around its
center is possible.  It is easy to see that all grasps listed in
Theorem~\ref{thm:char-ortho} are form closure grasps in this sense.
The grasp caused by a 4-pinning, however, is not a form-closure grasp,
as $\ell_{0}$ can be moved infinitesimally in the quadric defined by
three of the lines.

\section*{Acknowledgments}

Part of this work was done during the \emph{Kyoto RIMS Workshop on
  Computational Geometry and Discrete Mathematics 2008}, the
\emph{Oberwolfach Seminar on Discrete Geometry 2008} and the
\emph{BIRS Workshop on Transversal and Helly-type Theorems in
  Geometry, Combinatorics and Topology 2009}. We thank the Research
Institute for Mathematical Sciences of Kyoto University, the
Mathematisches Forschunginstitut Oberwolfach and the Banff
International Research Station for their hospitality and support. We
also thank Andreas Holmsen and Sang Won Bae for helpful discussions.

\bibliographystyle{plain}
\bibliography{polypin}

\end{document}